\documentclass[11pt,twoside]{article}
\usepackage[utf8]{inputenc}
\usepackage[english]{babel}
\usepackage{fancyhdr}
\usepackage{amssymb,amsmath,amsthm}
\usepackage[bookmarks]{hyperref}
\usepackage{caption}
\usepackage{mathtools}
\usepackage{booktabs}
\usepackage{algorithm}
\usepackage{algorithmic}
\usepackage{xcolor}
\usepackage{url}
\usepackage{bigstrut}
\usepackage[inline]{enumitem}
\usepackage[square,numbers]{natbib}

\theoremstyle{plain}
\newtheorem{theorem}{Theorem}[section]
\newtheorem{lemma}[theorem]{Lemma}
\newtheorem{corollary}[theorem]{Corollary}
\newtheorem{proposition}[theorem]{Proposition}
\newtheorem{remark}{Remark}

\newcommand{\R}{\ensuremath{\mathbb R}}
\newcommand{\A}{\ensuremath{\mathcal A}}
\newcommand{\N}{\ensuremath{\mathcal N}}
\newcommand{\ASZSL}{\texttt{AS-ZSL}}
\newcommand{\MVP}{Strategy MVP}
\newcommand{\AC}{Strategy AC2CD}
\newcommand{\compCL}{\texttt{compCL}}
\newcommand{\QP}{\texttt{QP}}
\newcommand{\zeroSum}{\texttt{zeroSum}}

\newcommand{\fijx}[3]{\ensuremath{f^{#1,#2}_{#3}}}
\newcommand{\eps}{\varepsilon}

\DeclarePairedDelimiter{\norm}{\lVert}{\rVert} 
\DeclarePairedDelimiter{\abs}{\lvert}{\rvert} 
\DeclareMathOperator*{\argmin}{Argmin} 
\DeclareMathOperator*{\argmax}{Argmax} 
\DeclareMathOperator*{\sgn}{sign}

\def\f0{\varphi}
\def\sigmau{\sigma_+}
\def\sigmav{\sigma_-}
\def\eps{\varepsilon}
\def\ihat{\hat \imath}
\def\jhat{\hat \jmath}
\def\jbar{\bar \jmath}

\hypersetup{
    colorlinks=false,
    urlcolor=black,
    pdfborder={0 0 0}
}

\newcommand{\email}[1]{E-mail: \href{mailto:#1}{\texttt{#1}}}

\begin{document}
\thispagestyle{plain}

\setcounter{page}{1}

{\centering
{\LARGE \bfseries A decomposition method for lasso problems with zero-sum constraint}

\bigskip\bigskip
Andrea Cristofari$^*$
\bigskip

}

\begin{center}
\small{\noindent$^*$Department of Civil Engineering and Computer Science Engineering, \\ University of Rome ``Tor Vergata'', Via del Politecnico, 1, 00133 Rome, Italy \\
\email{andrea.cristofari@uniroma2.it}\\
}
\end{center}

\bigskip\par\bigskip\par
\noindent \textbf{Abstract.}
In this paper, we consider lasso problems with zero-sum constraint,
commonly required for the analysis of compositional data in high-dimensional spaces.
A novel algorithm is proposed to solve these problems, combining a tailored active-set technique, to identify the zero variables in the optimal solution,
with a 2-coordinate descent scheme.
At every iteration, the algorithm chooses between two different strategies:
the first one requires to compute the whole gradient of the smooth term of the objective function and is more accurate in the active-set estimate,
while the second one only uses partial derivatives and is computationally more efficient.
Global convergence to optimal solutions is proved
and numerical results are provided on synthetic and real datasets, showing the effectiveness of the proposed method.
The software is publicly available.



\section{Introduction}
In this paper, we address the following optimization problem:
\begin{equation}\label{prob}
\begin{split}
& \min \frac 12 \norm{Ax-y}^2 + \lambda \norm x_1 \\
& \sum_{i=1}^n x_i = 0,
\end{split}
\end{equation}
where $\norm{\cdot}$ denotes the Euclidean norm, $\norm{\cdot}_1$ denotes the $\ell_1$ norm,
$x \in \R^n$ is the variable vector 
and $A \in \R^{m \times n}$, $y \in \R^m$, $\lambda \in \R$ are given, with $\lambda \ge 0$.

Problem~\eqref{prob} is an extension of the well known lasso problem~\cite{tibshirani:1996}, imposing the \textit{zero-sum constraint} $\sum_{i=1}^n x_i = 0$.
This constraint  is required in some regression models for \textit{compositional data}, i.e., for data representing percentages, or proportions, of a whole.
Applications with data of this type frequently arise in many different fields, such as geology, biology, ecology and economics.
For example, in microbiome analysis, datasets are usually normalized and result in compositional data~\cite{gloor:2017,shi:2016}.

Let us briefly review the role of the zero-sum constraint in regression analysis for compositional data.
Assume we are given a response vector $y \in \R^m$ and an $m \times n$ matrix $Z$ of covariates,
where every row of $Z$ is a sample.
By definition, compositional data are vectors whose components are non-negative and sum to~$1$,
so we can assume each row of $Z$ belong to the positive simplex
$\Delta^+_{n-1} = \{z \in \R^n \colon \sum_{i=1}^n z_i = 1, \, z > 0\}$.
Due to the constrained form of the sample space, standard statistical tools designed for unconstrained data cannot be applied to
build a regression model.
To overcome this issue, \textit{log-contrast models} were proposed in~\cite{aitchison:1982,aitchison:1984},
using the following log-ratio transformation of data:
\[
W_{ij} = \log\Bigl(\frac{Z_{ij}}{Z_{ir}}\Bigr), \quad i \in \{1,\ldots,m\}, \quad j \in \{1,\ldots,n\} \setminus \{r\},
\]
where the $r$th component is referred to as \textit{reference component}.
Denoting by $W_{\setminus r} \in \R^{m \times (n-1)}$ the matrix with entries $W_{ij}$,
a linear log-contrast model is then obtained as follows:
\[
y = W_{\setminus r} \, x_{\setminus r} + \eps,
\]
where $x_{\setminus r} = \begin{bmatrix} x_1 & \ldots & x_{r-1} & x_{r+1} & \ldots & x_n \end{bmatrix}^T \in \R^{n-1}$
is the vector of regression coefficients and $\eps \in \R^m$ is a vector of independent noise with mean zero and finite variances.
Using the definition of $W_{\setminus r}$, we can write
\[
y_i = \sum_{j \ne r} x_j \log{(Z_{ij})}  - \sum_{j \ne r} x_j \log(Z_{ir}) + \eps_i, \quad i = 1,\ldots,m.
\]
Hence, introducing $x_r = -\sum_{j \ne r} x_j$, we can express the above model in the following symmetric form:
\begin{equation}\label{log_contrast_model}
y = A x + \eps, \quad \text{subject to } \sum_{i=1}^n x_i = 0,
\end{equation}
where $A \in \R^{m \times n}$ is the matrix with $A_{ij} = \log(Z_{ij})$ in position $(i,j)$.
Note that the zero-sum constraint has a central role, since it allows the response $y$ to be expressed as a linear combination of log-ratios.

Starting from~\eqref{log_contrast_model}, an $\ell_1$-regularized least-square formulation with zero-sum constraint was first considered in~\cite{lin:2014}
for variable selection in high-dimensional spaces, leading to the optimization problem~\eqref{prob}.
The resulting estimator was shown in~\cite{lin:2014} to be \textit{model selection consistent}
as well as to ensure \textit{scale invariance}, \textit{permutation invariance} and \textit{selection invariance}.
Moreover, in~\cite{altenbuchinger:2017} it was shown that model~\eqref{prob} is \textit{proportional reference point insensitive},
allowing to overcome some issues related to the choice of a reference point in molecular measurements.

Well established algorithms can be used to solve~\eqref{prob},
exploiting the peculiar structure of the problem.
In this fashion, an augmented Lagrangian scheme with the subproblem solved by cyclic coordinate descent was proposed in~\cite{lin:2014},
while a coordinate descent strategy based on random selection of variables was proposed in~\cite{altenbuchinger:2017}.
Moreover, considering more general forms of constrained lasso, an approach based on quadratic programming and an ADMM method were analyzed in~\cite{gaines:2018},
a semismooth Newton augmented Lagrangian method was proposed in~\cite{deng:2020} and path algorithms were designed in~\cite{gaines:2018,jeon:2020,tibshirani:2011}.

In this paper we propose a decomposition algorithm, named Active-Set Zero-Sum-Lasso (\ASZSL), to efficiently solve problem~\eqref{prob} in a large-scale setting.
The first ingredient of our approach is an \textit{active-set technique} to identify the zero variables in the optimal solution,
which is expected to be sparse due to the $\ell_1$ regularization.
The second ingredient is a \textit{2-coordinate descent scheme} to update the variables estimated to be non-zero in the final solution.
In more detail, we define two different strategies:
the first one uses the whole gradient of the smooth term of the objective function and is more accurate in the identification of the zero variables,
while the second one only needs partial derivatives and is computationally more efficient.
To balance computational efficiency and accuracy, in the algorithm we use a simple rule to choose between the two strategies,
based on the progress in the objective function.
The proposed method is proved to converge to optimal solutions and is shown to perform well in the numerical experiments,
compared to other approaches from the literature.
The~\ASZSL\ software is available at~\url{https://github.com/acristofari/as-zsl}.

The rest of the paper is organized as follows.
In Section~\ref{sec:opt} we introduce the notation and give some optimality results for problem~\eqref{prob}, in Section~\ref{sec:alg} we present the~\ASZSL\ algorithm,
in Section~\ref{sec:conv} we establish global convergence of~\ASZSL\ to optimal solutions, in Section~\ref{sec:res} we show numerical results on synthetic and real datasets,
in Section~\ref{sec:concl} we finally draw some conclusions.
Moreover, in Appendix~\ref{app:subproblem} we report some details on the subproblem we need to solve in the algorithm.

\section{Preliminaries and optimality results}\label{sec:opt}
In this section, we introduce the notation and give some optimality results for problem~\eqref{prob}.
First note that we did not include an intercept term in model~\eqref{log_contrast_model},
since it can be omitted if the vector $y$ and the columns of $A$ have mean zero.


\subsection{Notation}
We indicate by $e$ the vector made of all ones (so that the zero-sum constraint $\sum_{i=1}^n x_i = 0$
can be rewritten as $e^T x = 0$).
We denote the Euclidean norm, the $\ell_1$ norm and the sup norm of a vector $v$ by $\norm v$, $\norm v_1$ and $\norm v_{\infty}$, respectively.
The maximum between a vector $v$ and $0$, denoted by $\max\{v,0\}$, is understood as a vector whose $i$th entry is $\max\{v_i,0\}$.
Given a matrix $M$, the element in position $(i,j)$ is indicated with $M_{ij}$, while $M^i$ is the $i$th column of $M$.

Moreover, we define
\[
\f0(x) = \dfrac 12 \norm{Ax-y}^2,
\]
so that the objective function of problem~\eqref{prob} can be expressed as
\[
f(x) = \f0(x) + \lambda \norm x_1.
\]
The gradient of $\f0$ is denoted by $\nabla \f0$.

\subsection{Optimality conditions}
Let us first rewrite~\eqref{prob} as a smooth problem with non-negativity constraints,
using a well known variable transformation~\cite{gaines:2018,schmidt:2007,tibshirani:1996}.
More specifically, we can introduce the variables $x_+,x_- \in \R^n$ in order to split $x$ into positive and negative components, i.e.,
$x = x_+ - x_-$, with $x_+ = \max\{x,0\}$ and $x_- = \max\{-x,0\}$.
We obtain the following reformulation:
\begin{equation}\label{prob_eq}
\begin{split}
& \min \f0(x_+-x_-) + \lambda e^T(x_++x_-) \\
& e^T(x_+-x_-) = 0 \\
& x_+ \ge 0 \\
& x_- \ge 0.
\end{split}
\end{equation}
Note that the number of variables has doubled, i.e., problem~\eqref{prob_eq} has $2n$ variables.
It is easy to obtain the following equivalence between problems~\eqref{prob} and~\eqref{prob_eq}.
\begin{lemma}\label{lemma:opt_preq}
Let $x^*$ be an optimal solution of problem~\eqref{prob}. Then, $(x_+^*,x_-^*)$ is an optimal solution of problem~\eqref{prob_eq}, where $x_+^* = \max\{x^*,0\}$ and $x_-^* = \max\{-x^*,0\}$.

Vice versa, if $(x_+^*,x_-^*)$ is an optimal solution of problem~\eqref{prob_eq}, then $x^* = x_+^* - x_-^*$ is an optimal solution of problem~\eqref{prob}.
\end{lemma}

Since problem~\eqref{prob_eq} is convex and all constraints are linear,
we can use KKT conditions and state that a point $(x_+^*,x_-^*)$ is an optimal solution of problem~\eqref{prob_eq} if and only if
there exist KKT multipliers $\sigmau^*, \sigmav^* \in \R^n$ and $\mu^* \in \R$
such that
\begin{align*}
& \nabla_{x_+} \f0(x_+^*-x_-^*) + \lambda e - \mu^* e - \sigmau^* = 0, \\
& \nabla_{x^-} \f0(x_+^*-x_-^*) + \lambda e + \mu^* e - \sigmav^* = 0, \\
& e^T(x_+^*-x_-^*) = 0, \\
& x_+^* \ge 0, \quad x_-^* \ge 0, \\
& \sigmau^* \ge 0, \quad \sigmav^* \ge 0, \\
& (\sigmau^*)^T x_+^* = 0, \quad (\sigmav^*)^T x_-^* = 0.
\end{align*}

From Lemma~\ref{lemma:opt_preq} and the fact that $\nabla_{x_+} \f0(x_+-x_-) = -\nabla_{x^-} \f0(x_+-x_-)$,
it follows that a point $x^*$ is an optimal solution of problem~\eqref{prob} if and only if there exist $\sigmau^*, \sigmav^* \in \R^n$ and $\mu^* \in \R$ such that
\begin{subequations}\label{kkt}
\begin{align}
& \nabla \f0(x^*) + \lambda e - \mu^* e - \sigmau^* = 0, \label{kkt_opt_prim_u} \\
& \nabla \f0(x^*) - \lambda e - \mu^* e + \sigmav^* = 0, \label{kkt_opt_prim_v} \\
& e^T x^* = 0, \\
& \sigmau^* \ge 0, \quad \sigmav^* \ge 0, \label{kkt_feas_dual} \\
& (\sigmau^*)_i \, x^*_i = 0, \quad \forall i \colon x^*_i > 0, \label{kkt_compl_u} \\
& (\sigmav^*)_i \, x^*_i = 0, \quad \forall i \colon x^*_i < 0. \label{kkt_compl_v}
\end{align}
\end{subequations}

The next theorem provides equivalent optimality conditions for problem~\eqref{prob}.
\begin{theorem}\label{th:opt}
Let $x^*$ be a feasible point of problem~\eqref{prob}. The following statements are equivalent:
\begin{enumerate}[label=(\alph*), leftmargin=*]
\item\label{th_opt_a} $x^*$ is an optimal solution of problem~\eqref{prob};
\item\label{th_opt_b} A scalar $\mu^*$ (the same as the one appearing in~\eqref{kkt}) exists such that
    \[
    \begin{cases}
    \nabla_i \f0(x^*) - \mu^* = \lambda, \quad & i \colon x^*_i < 0, \\
    \nabla_i \f0(x^*) - \mu^* = -\lambda, \quad & i \colon x^*_i > 0, \\
    \abs{\nabla_i \f0(x^*) - \mu^*} \le \lambda, \quad & i \colon x^*_i = 0;
    \end{cases}
    \]
\item\label{th_opt_c} $\eta^{\text{min}}(x^*) \ge \eta^{\text{max}}(x^*)$, where
    \begin{align*}
    \eta^{\text{min}}(x) & = \min_{i=1,\ldots,n} \bigl\{\nabla_i \f0(x) + \bigl(2\min\{\sgn(x_i),0\}+1\bigr)\lambda\bigr\}, \\
    \eta^{\text{max}}(x) & = \max_{i=1,\ldots,n} \bigl\{\nabla_i \f0(x) + \bigl(2\max\{\sgn(x_i),0\}-1\bigr)\lambda\bigr\}.
    \end{align*}
\end{enumerate}
\end{theorem}

\begin{proof}
We show the following implications.
\begin{enumerate}[align=left]
\item[$\bullet \; \textbf{\ref{th_opt_a}} \Rightarrow \textbf{\ref{th_opt_b}}$.]
    If $x^*$ is an optimal solution, then~\eqref{kkt} holds.
    If $x^*_i<0$, from~\eqref{kkt_feas_dual}--\eqref{kkt_compl_v} we have $(\sigmav^*)_i=0$ and, from~\eqref{kkt_opt_prim_v}, it follows that $\nabla_i \f0(x^*) - \mu^* = \lambda$.
    By the same arguments, if $x^*_i>0$ we have $(\sigmau^*)_i=0$ and $\nabla_i \f0(x^*) - \mu^* = -\lambda$.
    If $x^*_i=0$, from~\eqref{kkt_opt_prim_u}, \eqref{kkt_opt_prim_v} and~\eqref{kkt_feas_dual} we have $-\lambda \le \nabla_i \f0(x^*) - \mu^* \le \lambda$.
\item[$\bullet \; \textbf{\ref{th_opt_b}} \Rightarrow \textbf{\ref{th_opt_a}}$.]
    If~\ref{th_opt_b} holds, then the KKT system~\eqref{kkt} is satisfied by setting
    $\sigmau^* = \nabla \f0(x^*) + \lambda e - \mu^* e$ and $\sigmav^* = -\nabla \f0(x^*) + \lambda e + \mu^* e$, implying that $x^*$ is an optimal solution.
\item[$\bullet \; \textbf{\ref{th_opt_b}} \Rightarrow \textbf{\ref{th_opt_c}}$.]
    Let us rewrite $\eta^{\text{min}}(x^*)$ and $\eta^{\text{max}}(x^*)$ as follows:
    \begin{subequations}\label{eta}
    \begin{align}
    \eta^{\text{min}}(x^*) & = \min\Bigl\{\min_{i \colon x^*_i \ge 0} \nabla_i \f0(x^*) + \lambda, \, \min_{i \colon x^*_i<0} \nabla_i \f0(x^*) - \lambda\Bigr\}, \\
    \eta^{\text{max}}(x^*) & = \max\Bigl\{\max_{i \colon x^*_i \le 0} \nabla_i \f0(x^*) - \lambda, \, \max_{i \colon x^*_i>0} \nabla_i \f0(x^*) + \lambda\Bigr\}.
    \end{align}
    \end{subequations}
    If~\ref{th_opt_b} holds, we have
    \[
    \nabla_i \f0(x^*) + \lambda
    \begin{cases}
    \ge \nabla_i \f0(x^*) - \lambda = \mu^*, \quad & i \colon x^*_i < 0, \\
    = \mu^*, \quad & i \colon x^*_i > 0, \\
    \ge \mu^*, \quad & i \colon x^*_i = 0, \\
    \end{cases}
    \]
    and
    \[
    \nabla_i \f0(x^*) - \lambda
    \begin{cases}
    = \mu^*, \quad & i \colon x^*_i < 0, \\
    \le \nabla_i \f0(x^*) + \lambda = \mu^*, \quad & i \colon x^*_i > 0, \\
    \le \mu^*, \quad & i \colon x^*_i = 0. \\
    \end{cases}
    \]
    Therefore, $\eta^{\text{min}}(x^*) \ge \mu^* \ge \eta^{\text{max}}(x^*)$.
\item[$\bullet \; \textbf{\ref{th_opt_c}} \Rightarrow \textbf{\ref{th_opt_b}}$.]
    Let us consider $\eta^{\text{min}}(x^*)$ and $\eta^{\text{max}}(x^*)$ written as in~\eqref{eta}.
    If~\ref{th_opt_c} holds, let $\mu^*$ be any number in $[\eta^{\text{max}}(x^*),\eta^{\text{min}}(x^*)]$.
    If $\{i \colon x^*_i<0\} \ne \emptyset$, we can write
    \[
    \begin{split}
    \min_{i \colon x^*_i<0} \nabla_i \f0(x^*) - \lambda & \le \max_{i \colon x^*_i<0} \nabla_i \f0(x^*) - \lambda \le \max_{i \colon x^*_i \le 0} \nabla_i \f0(x^*) - \lambda \\
    & \le \eta^{\text{max}}(x^*) \le \mu^* \le \eta^{\text{min}}(x^*) \le \min_{i \colon x^*_i<0} \nabla_i \f0(x^*) - \lambda.
    \end{split}
    \]
    Therefore, all the inequalities in the above chain are actually equalities, implying that
    $\mu^* = \min_{i \colon x^*_i<0} \nabla_i \f0(x^*) - \lambda = \max_{i \colon x^*_i<0} \nabla_i \f0(x^*) - \lambda$. Namely,
    \[
    \mu^* = \nabla_i \f0(x^*) - \lambda, \quad \forall i \colon x^*_i < 0,
    \]
    and the first condition of~\ref{th_opt_b} is satisfied.
    Similarly, if $\{i \colon x^*_i>0\} \ne \emptyset$, we can write
    \[
    \begin{split}
    \max_{i \colon x^*_i>0} \nabla_i \f0(x^*) + \lambda & \ge \min_{i \colon x^*_i>0} \nabla_i \f0(x^*) + \lambda \ge \min_{i \colon x^*_i \ge 0} \nabla_i \f0(x^*) + \lambda \\
    & \ge \eta^{\text{min}}(x^*) \ge \mu^* \ge \eta^{\text{max}}(x^*) \ge \max_{i \colon x^*_i>0} \nabla_i \f0(x^*) + \lambda,
    \end{split}
    \]
    implying that
    \[
    \mu^* = \nabla_i \f0(x^*) + \lambda, \quad \forall i \colon x^*_i > 0,
    \]
    and the second condition of~\ref{th_opt_b} is satisfied.
    Finally, if $\{i \colon x^*_i = 0\} \ne \emptyset$, we can write
    \[
    \begin{split}
    \min_{i \colon x^*_i=0} \nabla_i \f0(x^*) + \lambda & \ge \min_{i \colon x^*_i \ge 0} \nabla_i \f0(x^*) + \lambda \ge
    \eta^{\text{min}}(x^*) \ge \mu^* \ge \eta^{\text{max}}(x^*) \\
    & \ge \max_{i \colon x^*_i \le 0} \nabla_i \f0(x^*) - \lambda \ge \max_{i \colon x^*_i=0} \nabla_i \f0(x^*) - \lambda,
    \end{split}
    \]
    implying that
    \[
    \nabla_i \f0(x^*) - \lambda \le \mu^* \le \nabla_i \f0(x^*) + \lambda, \quad \forall i \colon x^*_i = 0,
    \]
    and the third condition of~\ref{th_opt_b} is satisfied.
\end{enumerate}
\end{proof}

Note that, by point~\ref{th_opt_c} of Theorem~\ref{th:opt}, we can define $\eta^{\text{max}}(x) - \eta^{\text{min}}(x)$
as a measure of optimality violation.
Moreover, now we derive an expression for the KKT multiplier $\mu^*$ appearing in Theorem~\ref{th:opt},
which will be used in Section~\ref{sec:alg} to define an active-set estimate.

\begin{theorem}\label{th:mu}
If $x^* \ne 0$ is an optimal solution of problem~\eqref{prob},
then the corresponding multiplier $\mu^*$ appearing in point~\ref{th_opt_b} of Theorem~\ref{th:opt} is given by
\[
\mu^* = \frac{\displaystyle{\sum_{i \colon x^*_i \ne 0}|x^*_i|^p\bigl(\nabla_i \f0(x^*)+\lambda \sgn(x^*_i)\bigr)}}{\displaystyle{\sum_{i \colon x^*_i \ne 0}|x^*_i|^p}},
\quad \forall p \ge 0.
\]
\end{theorem}

\begin{proof}
Consider the KKT multiplier $\mu^*$ appearing in~\eqref{kkt}, which is the same as the one appearing in point~\ref{th_opt_b} Theorem~\ref{th:opt}.
From~\eqref{kkt_opt_prim_u} and~\eqref{kkt_opt_prim_v}, we obtain
\begin{align*}
\sigmau^* & = \nabla \f0(x^*) + \lambda e - \mu^*e, \\
\sigmav^* & = -\nabla \f0(x^*) + \lambda e + \mu^*e.
\end{align*}
Using~\eqref{kkt_feas_dual}--\eqref{kkt_compl_v}, it follows that $\mu^*$ is the unique optimal solution of the following one-dimensional strictly convex problem,
for all $p \ge 0$:
\[
\min_{\mu} \frac 12 \Biggl[ \, \sum_{i \colon x^*_i > 0} (\nabla_i \f0(x^*) + \lambda - \mu)^2 (x^*_i)^p
 + \sum_{i \colon x^*_i < 0} (-\nabla_i \f0(x^*) + \lambda + \mu)^2 (-x^*_i)^p\biggr].
\]
To compute $\mu^*$ as the minimizer of the above univariate function, we have to set its derivative to $0$, that is,
\[
\mu^* = \frac{\displaystyle{\sum_{i \colon x^*_i > 0} (x^*_i)^p (\nabla_i \f0(x^*)+\lambda) + \sum_{i \colon x^*_i < 0} (-x^*_i)^p (\nabla_i \f0(x^*)-\lambda)}}
{\displaystyle{\sum_{i \colon x^*_i > 0} (x^*_i)^p + \sum_{i \colon x^*_i < 0} (-x^*_i)^p}},
\]
leading to the expression given in the assertion.
\end{proof}

Let us conclude this section by showing a regularity property of the non-zero feasible points of problem~\eqref{prob},
which will be useful in the global convergence analysis of the proposed algorithm.

\begin{proposition}\label{prop:cond_mu}
Let $\bar x$ be a feasible point of problem~\eqref{prob} such that $\bar x_j \ne 0$ for an index $j \in \{1,\ldots,n\}$.
If there exists $I \subseteq \{1,\ldots,n\}$ such that
\[
f(\bar x) = \min_{\xi \in \R} f(\bar x + \xi (e_i - e_j)), \quad \forall i \in I,
\]
then there exists $\bar \mu \in \R$ such that
\begin{equation}\label{cond_mu}
\begin{cases}
\nabla_i \f0(\bar x) - \bar \mu = \lambda, \quad & i \in I \colon \bar x_i < 0, \\
\nabla_i \f0(\bar x) - \bar \mu = -\lambda, \quad & i \in I \colon \bar x_i > 0, \\
\abs{\nabla_i \f0(\bar x) - \bar \mu} \le \lambda, \quad & i \in I \colon \bar x_i = 0.
\end{cases}
\end{equation}
\end{proposition}

\begin{proof}
For any point $x \in \R^n$ and a (non-zero) direction $d \in \R^n$, let $f'(x;d)$ be the directional derivative of $f$ at $x$ along $d$.
We have
\begin{equation}\label{dir_der}
\begin{split}
f'(x;d) & = \lim_{\eps \to 0^+} \frac{f(x+\eps d)-f(x)}{\eps}
          = \nabla \f0(x)^T d + \lambda \lim_{\eps \to 0^+} \frac{\norm{x+\eps d}_1-\norm x_1}{\eps} \\
        & = \nabla \f0(x)^T d + \lambda \Biggl(\,\sum_{i=1}^n \sgn(x_i) \, d_i + \sum_{i \colon x_i = 0} \abs{d_i}\Biggr).
\end{split}
\end{equation}
Since $f(\bar x) = \min_{\xi \in \R} f(\bar x + \xi (e_i - e_j))$ for all $i \in I$ by hypothesis,
from the convexity of $f$ we get
\begin{equation}\label{dir_der_non_neg}
f'(\bar x; e_i - e_j) \ge 0 \quad \text{and} \quad f'(\bar x; e_j - e_i) \ge 0, \quad \forall i \in I.
\end{equation}
Without loss of generality, we assume that $\bar x_j>0$ (the proof for the case $\bar x_j<0$ is identical, except for minor changes).
Using~\eqref{dir_der}, we have
\begin{align}
f'(\bar x; e_i-e_j) = &
\begin{cases}
\nabla_i \f0(\bar x) - \nabla_j \f0(\bar x) - 2 \lambda, \quad & \text{if } \bar x_i < 0, \\
\nabla_i \f0(\bar x) - \nabla_j \f0(\bar x), \quad & \text{if } \bar x_i \ge 0;
\end{cases} \label{dir_der_1} \\
f'(\bar x; e_j-e_i) = &
\begin{cases}
\nabla_j \f0(\bar x) - \nabla_i \f0(\bar x) + 2 \lambda, \quad & \text{if } \bar x_i \le 0, \\
\nabla_j \f0(\bar x) - \nabla_i \f0(\bar x), \quad & \text{if } \bar x_i > 0.
\end{cases} \label{dir_der_2}
\end{align}
Therefore, in view of~\eqref{dir_der_non_neg}, for all $i \in I$ such that $\bar x_i \ne 0$ we can write
\[
\nabla_i \f0(\bar x) =
\begin{cases}
\nabla_j \f0(\bar x) + 2 \lambda, \quad & i \in I \colon \bar x_i<0, \\
\nabla_j \f0(\bar x), \quad & i \in I \colon \bar x_i>0.
\end{cases}
\]
So, we can set $\bar \mu = \nabla_j \f0(\bar x) + \lambda$ and the first two conditions of~\eqref{cond_mu} are satisfied.
Moreover, with this choice of $\bar \mu$ we have $\abs{\nabla_i \f0(\bar x) - \bar \mu} = \abs{\nabla_i \f0(\bar x) - \nabla_j \f0(\bar x) - \lambda}$.
So, to show that also the last condition of~\eqref{cond_mu} holds, we have to show that
\[
0 \le \nabla_i \f0(\bar x) - \nabla_j \f0(\bar x) \le 2 \lambda, \quad i \in I \colon \bar x_i = 0.
\]
This follows from~\eqref{dir_der_non_neg}, \eqref{dir_der_1} and~\eqref{dir_der_2}.
\end{proof}

\begin{remark}
In Proposition~\ref{prop:cond_mu}, if $I = \{1,\ldots,n\}$, then $\bar x$ is an optimal solution of problem~\eqref{prob},
according to point~\ref{th_opt_b} of Theorem~\ref{th:opt}.
\end{remark}

\section{The algorithm}\label{sec:alg}
In this section we propose a decomposition algorithm, named Active-Set Zero-Sum-Lasso (\ASZSL),
to efficiently solve problem~\eqref{prob}.

The underlying assumptions  motivating our approach are that
\textit{the optimal solutions are sparse}, due to the sparsity-promoting $\ell_1$ regularization,
and that \textit{the problem dimension is large}.
To face these issues, the proposed method relies on two main ingredients:
\begin{enumerate}[label=(\roman*)]
\item an \textit{active-set technique} to estimate the zero variables in the optimal solution,
\item a \textit{2-coordinate descent scheme} to update only the variables estimated to be non-zero in the optimal solution.
\end{enumerate}

In the field of constrained optimization, several active-set techniques were proposed to identify the active (or binding) constraints, see, e.g., \cite{andretta:2005,birgin:2002,cristofari:2017,cristofari:2020,cristofari:2022augmented,facchinei:1995,facchinei:1998,hager:2020,hager:2006,scheinberg:2006,schwartz:1997}.
Active-set techniques were successfully used also to identify the zero variables in $\ell_1$-regularized problems
\cite{byrd:2016,desantis:2016,keskar:2016,solntsev:2015,wen:2010,wen:2012}
and in $\ell_1$-constrained problems \cite{cristofari:2022minimization}.
Moreover, \textit{screening rules} were proposed to identify variables that can be discarded in lasso-type problems
\cite{ghaoui:2010,tibshirani:2012,xiang:2011,xiang:2012,wang:2013}
and \textit{shrinking techniques} have been widely used in some algorithms for machine learning problems to fix subsets of variables
\cite{chang:2011,hsieh:2008,joachims:1999,yuan:2010}.

Our approach makes use of the so-called \textit{multiplier functions}, which will be defined later,
adapting the active-set technique originally proposed in~\cite{facchinei:1995} for non-linearly constrained problems.

To explain our active-set technique,
let us first define the following index sets for any optimal solution $x^*$ of problem~\eqref{prob}:
\begin{align*}
\bar \A(x^*) & = \{i \colon x^*_i = 0\}, \\
\bar \N(x^*) & = \{1,\ldots,n\} \setminus \bar \A(x^*) = \{i \colon x^*_i \ne 0\}.
\end{align*}
We say that $\bar \A(x^*)$ and $\bar \N(x^*)$ represent the \textit{active set} and the \textit{non-active set}, respectively, at $x^*$.
In any point $x^k$ produced by the algorithm, we get estimates of $\bar \A(x^*)$ and $\bar \N(x^*)$
exploiting point~\ref{th_opt_b} of Theorem~\ref{th:opt}.
Namely, we set the vector $\pi^k$ as an approximation of $\nabla \f0(x^k) - \mu^* e$ and define
\begin{subequations}\label{estimate}
\begin{align}
\A^k & = \{i \colon x^k_i = 0, \, |\pi^k_i|\le \lambda \}, \label{active_set_estimate} \\
\N^k & = \{1,\ldots,n\} \setminus \A^k, \label{nonactive_set_estimate}
\end{align}
\end{subequations}
as estimates of $\bar \A(x^*)$ and $\bar \N(x^*)$, respectively.

Once $\A^k$ and $\N^k$ have been computed, we want to move only the variables estimated to be non-zero in the final solution,
i.e., the variables in $\N^k$, thus working in a lower dimensional space.
Moreover, to efficiently address large-scale problems, a 2-coordinate descent scheme is used for the variable update,
that is, we move two coordinates at a time (this is the minimum number of variables we can move to maintain feasibility, due to the equality constraint).

Given a feasible point $x^k \ne 0$ produced by the algorithm, in the sequel we define two possible strategies to compute $\pi^k$ in~\eqref{estimate}
and to update the variables.
The first strategy uses the whole gradient $\nabla \f0$ and is more accurate in the active-set estimate,
while the second strategy only uses partial derivatives and is computationally more efficient.
The rationale is trying to balance accuracy and computational efficiency,
in order to calculate the whole gradient vector $\nabla \f0$ only when needed.
In particular, in our algorithm we never compute the matrix $A^TA$, since this is impractical for large dimensions.
So, if the residual is known, $\mathcal O(m)$ operations are needed to compute a single partial derivative,
while $\mathcal O(mn)$ operations are needed to compute $\nabla \f0$.

\subsection{\MVP}\label{subsec:mvp}
In this first strategy, we need to compute the whole gradient $\nabla \f0(x^k)$.
Then, to obtain $\pi^k$ in~\eqref{estimate}, we use an approximation of the KKT multiplier $\mu^*$ by means of the so called multiplier functions.
More precisely, given a neighborhood $X$ of an optimal solution $x^*$,
we say that $\mu \colon X \to \R$ is a multiplier function if it is continuous in $x^*$ and such that
$\mu(x^*) = \mu^*$.
A class of multiplier functions can be straightforwardly obtained from Theorem~\ref{th:mu}. Namely, for all $x \ne 0$, we can define
\begin{equation}\label{mult_funct}
\mu(x) = \frac{\displaystyle{\sum_{i=1}^n|x_i|^p\bigl(\nabla_i \f0(x)+\lambda \sgn(x_i)\bigr)}}{\displaystyle{\sum_{i=1}^n|x_i|^p}}, \quad p>0.
\end{equation}
Once $\nabla \f0(x^k)$ and $\mu(x^k)$ have been computed, we can set
\begin{equation}\label{pi_k}
\pi^k = \nabla \f0(x^k) - \mu(x^k)e.
\end{equation}
With this choice of $\pi^k$ we have the following identification property,
ensuring that, if we are sufficiently close to an optimal solution $x^*$, then $i \in \A^k \Rightarrow x^*_i = 0$,
while the inverse implication (i.e., $x^*_i = 0 \Rightarrow i \in \A^k$) holds if $|\nabla_i \f0(x^*) - \mu^*| < \lambda$.

\begin{proposition}\label{prop:estimate}
Let $\A^k$ and $\N^k$ be defined as in~\eqref{estimate}, with $\pi^k$ computed as in~\eqref{pi_k}
and $\mu(x^k)$ computed as in~\eqref{mult_funct}.
Then, for any optimal solution $x^*$ of problem~\eqref{prob}, there exists a neighborhood $\mathcal B(x^*)$ such that
\[
\bar \A^+(x^*) \subseteq \A^k \subseteq \bar \A(x^*), \quad \forall x^k \in \mathcal B(x^*),
\]
where $\bar \A^+(x^*) = \{i \colon x^*_i = 0, |\nabla_i \f0(x^*) - \mu^*| < \lambda\}$ and $\mu^*$ is the KKT multiplier.
\end{proposition}

\begin{proof}
The inclusion $\A^k \subseteq \bar \A(x^*)$ is trivial, since $x^*_i \ne 0$ implies that, for all $x^k$ in a neighborhood of $x^*$,
we have $x^k_i \ne 0$ and then $i \notin \A^k$.
The inclusion $\bar \A^+(x^*) \subseteq \A^k$ follows from the continuity of $\nabla \f0(x)$ and the continuity of $\mu(x)$,
recalling that $\mu(x^*) = \mu^*$.
\end{proof}

To update the variables in $\N^k$, we use a 2-coordinate descent scheme based on the so called \textit{maximal violating pair} (MVP),
i.e., we move the two variables that most violate an optimality measure.
In the literature,
similar choices were considered for singly linearly constrained problems when the objective function is smooth,
using proper optimality measures (see, e.g.,~\cite{beck:2014,joachims:1999,lin:2001,palagi:2005,platt:1999}).
In our case, using point~\ref{th_opt_c} of Theorem~\ref{th:opt}, a maximal violating pair in $\N^k$ is defined as any pair
$(\ihat,\jhat)$ such that
\begin{subequations}\label{mvp}
\begin{align}
\ihat & \in \argmin_{i \in \N^k} \bigl\{\nabla_i \f0(x^k) + \bigl(2\min\{\sgn(x^k_i),0\}+1\bigr)\lambda\bigr\}, \\
\jhat & \in \argmax_{j \in \N^k} \bigl\{\nabla_j \f0(x^k) + \bigl(2\max\{\sgn(x^k_j),0\}-1\bigr)\lambda\bigr\}.
\end{align}
\end{subequations}

The next feasible point $x^{k+1}$ is then obtained by minimizing $f$ with respect to $x_{\ihat}$ and $x_{\jhat}$,
keeping all the other variables fixed in $x^k$, that is,
\[
x^{k+1} \in \argmin \{f(x) \colon e^T x = 0, \, x_h = x^k_h, h \in \{1,\ldots,n\} \setminus \{\ihat,\jhat\}\}.
\]
Equivalently,
\begin{equation}\label{subprob_mvp}
\begin{split}
& x^{k+1} = x^k + \xi^*(e_{\ihat}-e_{\jhat}), \\
& \xi^* \in \argmin_{\xi \in \R} \{f(x^k + \xi(e_{\ihat}-e_{\jhat}))\}.
\end{split}
\end{equation}

\subsection{\AC}\label{subsec:ac2cd}
This second strategy does not require to compute the whole gradient $\nabla \f0$.
In~\eqref{estimate} we simply set $\pi^k = \pi^{k-1}$ and then we update the variables in $\N^k$ by means of an \textit{almost cyclic rule}.
In particular, we extend a \mbox{2-coordinate} descent method, named AC2CD, proposed
in~\cite{cristofari:2019} and further analyzed in~\cite{cristofari:2022active}.
Considering singly linearly constrained problems with lower and upper bound on the variables,
the method proposed in~\cite{cristofari:2019} chooses two variables at a time, such that
one of them must be ``sufficiently far'' from the lower and the upper bound in some points produced by the algorithm,
while the other one is picked cyclically.
In order to adapt this approach to our setting, we first have to choose a variable index $j(k)$ such that
\begin{equation}\label{jk}
|x^k_{j(k)}| \ge \tau \norm{x^k}_{\infty}, \quad j(k) \in \N^k,
\end{equation}
where $\tau \in (0,1]$ is a fixed parameter.
Namely, we require $x^k_{j(k)}$ to be ``sufficiently different'' from zero.
Then, we start a cycle of inner iterations where we update two variables at a time.
In particular, first we set $z^{k,1} = x^k$ and choose a permutation $p^k_1,\ldots,p^k_{|\N^k|}$ of $\N^k$,
then we select one index $p^k_i$ at a time in a cyclic fashion, with $i = 1,\ldots,|\N^k|$,
and consider the index pair $(p^k_i,j(k))$.
We compute $z^{k,i+1}$ by minimizing $f(z)$ with respect to $z_{p^k_i}$ and $z_{j(k)}$,
keeping all the other variables fixed in $z^{k,i}$.
Namely,
\[
z^{k,i+1} \in \argmin \{f(z) \colon e^T z = 0, \, z_h = z^{k,i}_h, h \in \{1,\ldots,n\} \setminus \{p^k_i,j(k)\}\}.
\]
Equivalently,
\begin{equation}\label{subprob_ac2cd}
\begin{split}
& z^{k,i+1} = z^{k,i} + \xi^*(e_{p^k_i}-e_{j(k)}), \\
& \xi^* \in \argmin_{\xi \in \R} \{f(z^{k,i} + \xi(e_{p^k_i}-e_{j(k)}))\}.
\end{split}
\end{equation}
After producing the points $z^{k,1}$, $z^{k,2}$, \ldots, $z^{k,|\N^k|+1}$,
we set the next feasible point $x^{k+1} = z^{k,|\N^k|+1}$.

\subsection{Choosing between the two strategies}\label{subsec:choosing}
We have seen that, one the one hand, \AC\ does not need the expensive computation of the whole gradient $\nabla \f0$,
but, on the other hand, we expect the active-set estimate used in \MVP\ to be more accurate
in a neighborhood of an optimal solution, according to Proposition~\ref{prop:estimate}.
In order to balance computational efficiency and accuracy,
we want to use \MVP\ only when we judge it is worthwhile to compute a new gradient and a new $\pi^k$ in~\eqref{estimate}.
This occurs when we observe no sufficient progress in the objective function.
In particular, given a parameter $\theta \in (0,1]$, at iteration $k$ we use \MVP\ if both $f(x^{k-1})-f(x^k) \le \theta \max\{f(x^{k-1}),1\}$
and \MVP\ was not used in $x^{k-1}$, otherwise we use \AC.
As to be shown below, eventually the algorithm alternates between the two strategies.

The scheme of the proposed~\ASZSL\ method is reported in Algorithm~\ref{alg:asac2cdl1}.

\begin{algorithm}[h!]
\footnotesize
\caption{Active-Set Zero-Sum-Lasso (\texttt{\ASZSL})}
\label{alg:asac2cdl1}
\begin{algorithmic}
\par\vspace*{0.1cm}
\item[]\hspace*{-0.1truecm}$\,\,\,0$\vspace{0.1truecm}\hspace*{0.1truecm} \textbf{Given} $\theta \in (0,1]$ and $\tau \in (0,1]$, set $x^0 = 0$ and $k=0$
\item[]\hspace*{-0.1truecm}$\,\,\,1$\vspace{0.1truecm}\hspace*{0.1truecm} \textbf{If} $x^0$ is not optimal, compute $x^1$ such that $f(x^1) < f(x^0)$, set $k=1$
and go to line 4
\item[]
\item[]\hspace*{-0.1truecm}$\,\,\,2$\vspace{0.1truecm}\hspace*{0.1truecm} \textbf{While} $x^k$ is not optimal
\item[]
\item[]\hspace*{-0.1truecm}$\,\,\,3$\vspace{0.1truecm}\hspace*{0.6truecm} \mbox{\textbf{If} $\dfrac{f(x^{k-1})-f(x^k)}{\max\{f(x^{k-1}),1\}} \le \theta $
    and \MVP\ was not used for $x^{k-1}$}
\item[]
\item[]\hspace*{-0.1truecm}$\,\,\,\,\,\,$\vspace{0.1truecm}\hspace*{1.1truecm} \underline{\textit{\MVP}}
\item[]\hspace*{-0.1truecm}$\,\,\,4$\vspace{0.1truecm}\hspace*{1.1truecm} Compute $\A^k$ and $\N^k$ as in~\eqref{estimate}, with $\pi^k$ computed as in~\eqref{pi_k}
\item[]\hspace*{-0.1truecm}$\,\,\,5$\vspace{0.1truecm}\hspace*{1.1truecm} Compute a maximal violating pair $(\ihat,\jhat)$ as in~\eqref{mvp}
\item[]\hspace*{-0.1truecm}$\,\,\,6$\vspace{0.1truecm}\hspace*{1.1truecm} Compute $x^{k+1}$ as in~\eqref{subprob_mvp}
\item[]
\item[]\hspace*{-0.1truecm}$\,\,\,7$\vspace{0.1truecm}\hspace*{0.6truecm} \textbf{Else}
\item[]
\item[]\hspace*{-0.1truecm}$\,\,\,\,\,\,$\vspace{0.1truecm}\hspace*{1.1truecm} \underline{\textit{\AC}}
\item[]\hspace*{-0.1truecm}$\,\,\,8$\vspace{0.1truecm}\hspace*{1.1truecm} Compute $\A^k$ and $\N^k$ as in~\eqref{estimate}, with $\pi^k = \pi^{k-1}$
\item[]\hspace*{-0.1truecm}$\,\,\,9$\vspace{0.1truecm}\hspace*{1.1truecm} Choose a variable index $j(k)$ satisfying $\eqref{jk}$
\item[]\hspace*{-0.1truecm}$10$\vspace{0.1truecm}\hspace*{1.1truecm} Set $z^{k,1} = x^k$
\item[]\hspace*{-0.1truecm}$11$\vspace{0.1truecm}\hspace*{1.1truecm} Choose a permutation $\{p^k_1,\ldots,p^k_{|\N^k|}\}$ of $\N^k$
\item[]\hspace*{-0.1truecm}$12$\vspace{0.1truecm}\hspace*{1.1truecm} \textbf{For} $i = 1,\ldots,|\N^k|$
\item[]\hspace*{-0.1truecm}$13$\vspace{0.1truecm}\hspace*{1.6truecm} Compute $z^{k,i+1}$ as in~\eqref{subprob_ac2cd}
\item[]\hspace*{-0.1truecm}$14$\vspace{0.1truecm}\hspace*{1.1truecm} \textbf{End for}
\item[]\hspace*{-0.1truecm}$15$\vspace{0.1truecm}\hspace*{1.1truecm} Set $x^{k+1} = z^{k,|\N^k|+1}$
\item[]
\item[]\hspace*{-0.1truecm}$16$\vspace{0.1truecm}\hspace*{0.6truecm} \textbf{End if}
\item[]
\item[]\hspace*{-0.1truecm}$17$\vspace{0.1truecm}\hspace*{0.6truecm} Set $k = k + 1$
\item[]\hspace*{-0.1truecm}$18$\vspace{0.1truecm}\hspace*{0.1truecm} \textbf{End while}
\end{algorithmic}
\end{algorithm}

\begin{remark}\label{rem:subprob}
Both \MVP\ and \AC\ require to solve a subproblem for every variable update,
given in~\eqref{subprob_mvp} and~\eqref{subprob_ac2cd}, respectively.
We see that each subproblem consists
in an exact minimization of $f$ over a direction of the form $\pm(e_i - e_j)$.
A finite procedure for this minimization is described in Appendix~\ref{app:subproblem}.
\end{remark}

\section{Convergence analysis}\label{sec:conv}
In this section, we show the convergence of \ASZSL\ to optimal solutions,
using some results on the proposed active-set estimate and standard arguments on block coordinate descent methods~\cite{grippo:2000,luo:1992,tseng:2001}.

First, we note that most results stated in the above sections require $x^k \ne 0$.
Indeed, \ASZSL\ ensures that
\[
x^k \ne 0, \quad \forall k \ge 1,
\]
since $x^0 = 0$ and
\begin{equation}\label{f_k_decr}
f(x^{k+1}) \le f(x^k) < f(x^0) = f(0), \quad \forall k \ge 1.
\end{equation}
As a consequence, for all $k \ge 1$,
\begin{itemize}
\item the multiplier function $\mu(x)$ given in~\eqref{mult_funct} is well defined at $x^k$;
\item $\N^k \ne \emptyset$, according to~\eqref{estimate}.
\end{itemize}

Since $f$ is continuous and coercive, also the following result follows from~\eqref{f_k_decr}, ensuring the convergence
of $\{f(x^k)\}$ and the existence of limit points for $\{x^k\}$.
\begin{lemma}\label{lemma:conv_f_x}
Let $\{x^k\}$ be an infinite sequence of points produced by \ASZSL.
Then,
\begin{enumerate}[label=(\roman*)]
\item\label{lim_f_lemma_conv_f_x} $\displaystyle{\lim_{k \to \infty} f(x^k) = f^* \in \R}$, with $f^*>0$;
\item\label{lim_x_lemma_conv_f_x} every subsequence $\{x^k\}_{K \subseteq \mathbb N}$ has limit points, each of them being feasible and different from zero.
\end{enumerate}
\end{lemma}


Moreover, for any iteration $k$ where \AC\ is used, from the instructions of the algorithm we have
\[
f(x^{k+1}) = f(z^{k,|\N^k|+1}) \le f(z^{k,|\N^k|}) \le \ldots \le f(z^{k,1}) = f(x^k).
\]
So, still using ~\eqref{f_k_decr} and the fact that $f$ is continuous and coercive,
we have the following result. It ensures, over a subsequence of points produced by \AC,
the existence of limit points and the convergence of the objective function.
\begin{lemma}\label{lemma:conv_f_z_ac2cd}
Let $\{x^k\}_{K \subseteq \mathbb N}$ be an infinite subsequence of points produced by \ASZSL\ such that \AC\ is used for all $k \in K$.
Then, for every fixed $i \in \{1,\ldots,\displaystyle{\min_{k \in K}|\N^k|}+1\}$,
\begin{enumerate}[label=(\roman*)]
\item\label{lim_f_lemma_conv_f_z_ac2cd} $\displaystyle{\lim_{\substack{k \to \infty \\ k \in K}} f(z^{k,i}) = \lim_{k \to \infty} f(x^k) = f^* \in \R}$, with $f^*>0$;
\item\label{lim_z_lemma_conv_f_z_ac2cd} the subsequence $\{z^{k,i}\}_K$ has limit points, each of them being feasible and different from zero.
\end{enumerate}
\end{lemma}

Now, we show that \ASZSL\ cannot select \MVP\ or \AC\ for
an arbitrarily large number of consecutive iterations.
In particular, provided $\{x^k\}$ is infinite, eventually the algorithm alternates between the two strategies.

\begin{proposition}\label{prop:strategy}
Let $\{x^k\}$ be an infinite sequence of points produced by \ASZSL.
Then, there exists $\bar k$ such that
\MVP\ is used for $k = \bar k, \bar k + 2, \bar k + 4, \ldots$
and \AC\ is used for $k = \bar k + 1, \bar k + 3, \bar k + 5, \ldots$.
\end{proposition}

\begin{proof}
From point~\ref{lim_f_lemma_conv_f_x} of Lemma~\ref{lemma:conv_f_x} and~\eqref{f_k_decr},
the sequence $\{f(x^k)\}$ converges to a value $f^*$ such that $f(x^k) \ge f^*$ for all $k \ge 0$.
It follows that
\[
\lim_{k \to \infty} \frac{f(x^{k-1})-f(x^k)}{\theta \max\{f(x^{k-1}),1\}} \le \lim_{k \to \infty} \frac{f(x^{k-1})-f(x^k)}{\theta \max\{f^*,1\}} = 0.
\]
Therefore, according to the test at line~3 of Algorithm~\ref{alg:asac2cdl1}, the algorithm alternates between \MVP\ and \AC\ infinitely
for sufficiently large~$k$.
\end{proof}

\subsection{Global convergence to optimal solutions}\label{subsec:conv}
Without loss of generality, for every index pair $(i,j)$ selected by \MVP\ or \AC,
now we require $f$ to be strictly convex along the directions $\pm(e_i-e_j)$.
In Appendix~\ref{app:strict_conv}, we show how this can be easily guaranteed.
Essentially, we just have to remove variables from the problem when we find identical columns in the matrix $A$.
Using this non-restrictive requirement, the following results show that $\norm{x^{k+1}-x^k} \to 0$.

\begin{proposition}\label{prop:lim_subseq_x_diff}
Let $\{x^k\}_{K_1 \subset \mathbb N}$ and $\{x^k\}_{K_2 \subset \mathbb N}$ be two infinite subsequences of points produced by~\ASZSL, such that
\MVP\ is used for all $k \in K_1$ and \AC\ is used for all $k \in K_2$. Then,
\begin{subequations}
\begin{align}
& \lim_{\substack{k \to \infty \\ k \in K_1}} \norm{x^{k+1}-x^k} = 0, \label{limxk1} \\
& \lim_{\substack{k \to \infty \\ k \in K_2}} \norm{x^{k+1} - x^k}
   = \lim_{\substack{k \to \infty \\ k \in K_2}} \sum_{i=1}^{|\N^k|} \norm{z^{k,i+1} - z^{k,i}} = 0. \label{limxk2}
\end{align}
\end{subequations}
\end{proposition}

\begin{proof}
By contradiction, assume that~\eqref{limxk1} does not hold.
It follows that there exists a subsequence $\{x^k\}_{K_3 \subseteq K_1}$
such that $\liminf_{k \to \infty, \, k \in K_3} \norm{x^{k+1} - x^k} > 0$.
By point~\ref{lim_x_lemma_conv_f_x} of Lemma~\ref{lemma:conv_f_x}, we can assume that both $\{x^k\}_{K_3}$ and $\{x^{k+1}\}_{K_3}$ converge
to feasible points (passing into a further subsequence if necessary), so that
\begin{equation}\label{lim_x}
\lim_{\substack{k \to \infty \\ k \in K_3}} x^k = x' \ne x'' = \lim_{\substack{k \to \infty \\ k \in K_3}} x^{k+1}.
\end{equation}
Since the set of variable indices is finite, we can also assume that the maximal violating pair $(\ihat,\jhat)$ is the same for all $k \in K_3$ (passing again into a further subsequence if necessary).
Using~\eqref{subprob_mvp}, \eqref{lim_x} and the continuity of $f$, we obtain
\[
x'' = x' + \xi^*(e_{\ihat}-e_{\jhat}), \quad \text{where} \quad \xi^* \in \argmin_{\xi \in \R} \{f(x' + \xi(e_{\ihat}-e_{\jhat}))\}.
\]
Namely, $x'$ and $x''$ belong to the line $\{x' + \xi(e_{\ihat}-e_{\jhat}), \, \xi \in \R\}$.
As said at the beginning of this subsection, without loss of generality here we require $f$ to be strictly convex along $\pm(e_{\ihat}-e_{\jhat})$
(in Appendix~\ref{app:strict_conv} it is shown how this can be easily guaranteed).
Since $x' \ne x''$, it follows that
%
\begin{equation}\label{ineq_fz}
f\Bigl(\dfrac{x'+x''}2\Bigr) < \dfrac 1 2 f(x')  + \dfrac 1 2 f(x'').
\end{equation}
Moreover, using again~\eqref{subprob_mvp} we can write
\[
f(x^{k+1}) \le f\Bigl(x^k + \frac 1 2 (x^{k+1}-x^k)\Bigr) = f\Bigl(\frac{x^k+x^{k+1}}2\Bigr) \le \frac{f(x^k)}2 + \frac{f(x^{k+1})}2,
\]
where the last inequality follows from the convexity of $f$.
Since $f(x^{k+1}) \le f(x^k)$, we obtain
\begin{equation}\label{ineq_exact_stepsize}
f(x^{k+1}) \le f\Bigl(\frac{x^k+x^{k+1}}2\Bigr) \le f(x^k).
\end{equation}
By~\eqref{lim_x} and point~\ref{lim_f_lemma_conv_f_x} of Lemma~\ref{lemma:conv_f_x}, using the continuity of $f$ we have that
\[
f(x') = \lim_{\substack{k \to \infty \\ k \in K_3}} f(x^k) =
\lim_{\substack{k \to \infty \\ k \in K_3}} f(x^{k+1}) = f(x'').
\]
Therefore, taking the limits in~\eqref{ineq_exact_stepsize}, we get
$f\Bigl(\dfrac{x'+x''}2\Bigr) = f(x') = f(x'')$, that is,
\[
f\Bigl(\dfrac{x'+x''}2\Bigr) = \dfrac 1 2 f(x')  + \dfrac 1 2 f(x''),
\]
contradicting~\eqref{ineq_fz}.

To show~\eqref{limxk2}, from the instructions of the algorithm we can write
$x^{k+1} - x^k = \sum_{i=1}^{|\N^k|} (z^{k,i+1} - z^{k,i})$ for all $k \in K_2$,
implying that
\[
\norm{x^{k+1} - x^k} \le \sum_{i=1}^{|\N^k|} \norm{z^{k,i+1} - z^{k,i}}, \quad \forall k \in K_2.
\]
So, to prove the desired result, we just have to show that
\[
\lim_{\substack{k \to \infty \\ k \in K_2}} \sum_{i=1}^{|\N^k|} \norm{z^{k,i+1} - z^{k,i}} = 0.
\]
Arguing by contradiction, assume that
there exist an index $\bar \imath \in \{1,\ldots,n\}$ and a subsequence $\{z^{k,\bar \imath}\}_{K_4 \subseteq K_2}$
such that $\liminf_{k \to \infty, \, k \in K_4} \norm{z^{k,\bar \imath+1} - z^{k,\bar \imath}} > 0$ for all $k \in K_4$.
By point~\ref{lim_z_lemma_conv_f_z_ac2cd} of Lemma~\ref{lemma:conv_f_z_ac2cd}, we can assume that both $\{z^{k,\bar \imath}\}_K$ and $\{z^{k,\bar \imath+1}\}_K$ converge
to feasible points (passing into a further subsequence if necessary).
Thus, we get a contradiction by the same arguments used to prove~\eqref{limxk1}.
\end{proof}

\begin{proposition}\label{prop:lim_x_diff}
Let $\{x^k\}$ be an infinite sequence of points produced by \ASZSL. Then,
\[
\lim_{k \to \infty} \norm{x^{k+1}-x^k} = 0.
\]
\end{proposition}

\begin{proof}
From Proposition~\ref{prop:strategy}, two infinite subsequences $\{x^k\}_{K_1 \subset \mathbb N}$ and  $\{x^k\}_{K_2 \subset \mathbb N}$ exist such that
\MVP\ is used for all $k \in K_1$ and \AC\ is used for all $k \in K_2$,
with $K_1 \cup K_2 = \mathbb N$.
Then, the desired result follows from Proposition~\ref{prop:lim_subseq_x_diff}.
\end{proof}

We are finally ready to show global convergence of \ASZSL\ to optimal solutions.

\begin{theorem}\label{th:conv}
Let $\{x^k\}$ be an infinite sequence of points produced by \ASZSL.
Then, $\{x^k\}$ has limit points and every limit point is an optimal solution of problem~\eqref{prob}.
\end{theorem}

\begin{proof}
From point~\ref{lim_x_lemma_conv_f_x} of Lemma~\ref{lemma:conv_f_x}, $\{x^k\}$ has limit points, each of them being feasible and different from zero.
Let $x^*$ be any of these limit points and let $\{x^k\}_{K \subseteq \mathbb N}$ be a subsequence converging to $x^*$.
In view of Proposition~\ref{prop:lim_x_diff}, $x^*$ is also a limit point of $\{x^{k+1}\}_K$ and of $\{x^{k-1}\}_K$. Namely,
\begin{equation}\label{lim_x_opt}
\lim_{\substack{k \to \infty \\ k \in K}} x^k = \lim_{\substack{k \to \infty \\ k \in K}} x^{k+1} = \lim_{\substack{k \to \infty \\ k \in K}} x^{k-1} = x^* \ne 0.
\end{equation}
Using Proposition~\ref{prop:strategy},
without loss of generality we can assume that \AC\ is used for all $k \in K$.
(In particular, if \AC\ is used for infinitely many $k \in K$, we can simply discard from $\{x^k\}_K$ the indices $k$ where \AC\ is not used.
On the contrary, if $\bar k$ exists such that \MVP\ is used for all $k \ge \bar k$, $k \in K$, then we can consider the subsequence $\{x^{k+1}\}_K$ instead.
This subsequence still converges to $x^*$ by~\eqref{lim_x_opt} and, for all sufficiently large $k \in K$,
\AC\ is used for all $k+1$ in view of Proposition~\ref{prop:strategy}.)

Since the set of variable indices is finite, for all $k \in K$ we can assume that $\A^k$, $\N^k$, $j(k)$ and $p^k_i$ are the same, that is,
\[
\A^k = \A, \qquad \N^k = \N, \qquad j(k) = \jbar, \qquad p^k_i = p_i, \quad \forall i \in \{1,\ldots,|\N|\}
\]
(passing into a further subsequence if necessary) and, by point~\ref{lim_z_lemma_conv_f_z_ac2cd} of Lemma~\ref{lemma:conv_f_z_ac2cd}, that
\begin{equation}\label{lim_z_i}
\lim_{\substack{k \to \infty \\ k \in K}} z^{k,i} = \bar z^i, \quad \forall i \in \{1,\ldots,|\N|+1\}
\end{equation}
(passing again into a further subsequence if necessary).

%
By continuity of $f$, we can take the limits in~\eqref{subprob_ac2cd} and, using~\eqref{lim_z_i}, for all $i \in \{1,\ldots,|\N|\}$ we have
\begin{equation}\label{z_i}
\begin{split}
& \bar z^{i+1} = \bar z^i + \xi^*(e_{p_i}-e_{\jbar}), \\
& \xi^* \in \argmin_{\xi \in \R} \{f(\bar z^i + \xi(e_{p_i}-e_{\jbar}))\}.
\end{split}
\end{equation}

Now, we show that
\begin{equation}\label{z_xstar}
\bar z^i = x^*, \quad \forall i \in \{1,\ldots,|\N|+1\}.
\end{equation}
From the instructions of the algorithm we have $z^{k,1} = x^k$, so we can write
$\norm{z^{k,i} - x^k} \le \sum_{h=1}^{i-1}\norm{z^{k,h+1}-z^{k,h}}$ for all $i \in \{1,\ldots,|\N|+1\}$.
Using~\eqref{limxk2}, it follows that
\begin{equation}\label{z_xk}
\lim_{\substack{k \to \infty \\ k \in K}} \norm{z^{k,i}-x^k} = 0, \quad \forall i \in \{1,\ldots,|\N|+1\}.
\end{equation}
Since $\norm{z^{k,i}-x^*} \le \norm{z^{k,i}-x^k}+ \norm{x^k-x^*}$, from~\eqref{lim_x_opt} and~\eqref{z_xk} we get
\[
\lim_{\substack{k \to \infty \\ k \in K}} \norm{z^{k,i}-x^*} = 0, \quad \forall i \in \{1,\ldots,|\N|+1\}.
\]
Using~\eqref{lim_z_i}, we thus obtain~\eqref{z_xstar}.

Therefore, from~\eqref{z_i} and~\eqref{z_xstar} it follows that
\begin{equation}\label{xstar}
f(x^*) = \min_{\xi \in \R} f(x^* + \xi(e_i-e_{\jbar})), \quad \forall i \in \N.
\end{equation}

Using~\eqref{lim_x_opt}, a real number $\eta>0$ exists such that $\|x^k\|_{\infty} \ge \eta/\tau$ for all sufficiently large $k \in K$,
where $\tau \in (0,1]$ is the parameter used in \ASZSL\ such that $|x^k_{\jbar}| \ge \tau \norm{x^k}_{\infty}$ (see line~9 of Algorithm~\ref{alg:asac2cdl1}).
Consequently, for all sufficiently large $k \in K$, we have $|x^k_{\jbar}| \ge \eta$, and then,
\begin{equation}\label{xstarj}
x^*_{\jbar} \ne 0.
\end{equation}
So, using~\eqref{xstar}, \eqref{xstarj} and Proposition~\ref{prop:cond_mu}, there exists $\mu^* \in \R$ such that
\begin{equation}\label{opt_N}
\begin{cases}
\nabla_i \f0(x^*) - \mu^* = \lambda, \quad & i \in \N \colon x^*_i < 0, \\
\nabla_i \f0(x^*) - \mu^* = -\lambda, \quad & i \in \N \colon x^*_i > 0, \\
\abs{\nabla_i \f0(x^*) - \mu^*} \le \lambda, \quad & i \in \N \colon x^*_i = 0.
\end{cases}
\end{equation}

Now, from the active-set estimate~\eqref{active_set_estimate} we observe that $i \in \A \Rightarrow x^k_i = 0$ for all $k \in K$.
Since $\{x^k\}_K \to x^*$ from~\eqref{lim_x_opt}, it follows that
\begin{equation}\label{A_subset}
\A \subseteq \{i \colon x^*_i = 0\}.
\end{equation}
Taking into account~\eqref{opt_N} and~\eqref{A_subset}, according to point~\ref{th_opt_b} of Theorem~\ref{th:opt}
we thus have to show that
\begin{equation}\label{opt_A}
\abs{\nabla_i \f0(x^*) - \mu^*} \le \lambda, \quad \forall i \in \A,
\end{equation}
in order to prove that $x^*$ is optimal and conclude the proof.
Note that, from~\eqref{opt_N} and~\eqref{A_subset},
we can write $|x^*_i|^p\bigl(\nabla_i \f0(x^*)+\lambda \sgn(x^*_i)\bigr) = |x^*_i|^p \mu^*$, $i = 1,\ldots,n$.
This implies that
\[
\sum_{i=1}^n |x^*_i|^p\bigl(\nabla_i \f0(x^*)+\lambda \sgn(x^*_i)\bigr) = \mu^* \sum_{i=1}^n |x^*_i|^p.
\]
So, using the definition of the multiplier function $\mu(x)$ given in~\eqref{mult_funct}, we get
\begin{equation}\label{mu_x_star}
\mu^* = \mu(x^*).
\end{equation}
Moreover, Proposition~\ref{prop:strategy} ensures that, for sufficiently large $k \in K$, \MVP\ is used at $x^{k-1}$,
implying that $\pi^k = \pi^{k-1} = \nabla \f0(x^{k-1}) - \mu(x^{k-1})e$ (see lines~8 and~4 of Algorithm~\ref{alg:asac2cdl1}).
Therefore, from the active-set estimate~\eqref{active_set_estimate}, for all sufficiently large $k \in K$ we can write
\[
\A = \{i \colon x^k_i = 0, \, |\nabla_i \f0(x^{k-1}) - \mu(x^{k-1})| \le \lambda\}.
\]
Since $\{x^k\}_K$ and $\{x^{k-1}\}_K$ converge to $x^*$ from~\eqref{lim_x_opt}, using the continuity of $\nabla \f0$, the continuity of
the multiplier function $\mu(x)$ and~\eqref{mu_x_star}, we can take the limits for $k \to \infty$, $k \in K$, and we finally get~\eqref{opt_A}.
\end{proof}

\section{Numerical results}\label{sec:res}
In this section, we report the numerical results obtained on synthetic and real datasets with compositional data.
We implemented \ASZSL\ in C++, using a MEX file to call the algorithm from Matlab.
The~\ASZSL\ software is available at~\url{https://github.com/acristofari/as-zsl}.

In our experiments, we use $p=1$ for the multiplier functions defined in~\eqref{mult_funct} and $\tau = 1$.
Moreover, $\theta$ is set to $10^{-2}$ at the beginning of the algorithm
and is gradually decreased to $10^{-6}$.
All tests were run on an Intel(R) Core(TM) i7-9700 with $16$ GB RAM memory.

We compared \ASZSL\ with the following algorithms:
\begin{itemize}
\item \compCL~\cite{lin:2014}, which solves~\eqref{prob} by the method of multiplier minimizing the augmented Lagrangian function by cyclic coordinate descent.
    The code was written in C++ and called from R.
    It was downloaded from \url{https://cran.r-project.org/package=Compack} as part of the R package Compack.
\item \QP~\cite{gaines:2018}, which solves the quadratic reformulation~\eqref{prob_eq}.
    Since $n>m$, a ridge term $10^{-4} \norm x ^2$ was added to the original objective function, as suggested in~\cite{gaines:2018}.
    The code was downloaded from \url{https://github.com/Hua-Zhou/SparseReg},
    it builds the problem via Matlab and uses the Gurobi Optimizer (version 9.5)~\cite{gurobi} for the minimization.
\item \zeroSum~\cite{altenbuchinger:2017}, which uses an extension of the random coordinate descent method to solve~\eqref{prob}.
    The code was written in C++ and called from R.
    It was downloaded from \url{https://github.com/rehbergT/zeroSum} as part of the R package zeroSum.
\end{itemize}

These algorithms were run with their default parameters and options, except those specified above.
We observe that both \compCL\ and \zeroSum\ use a (block) coordinate descent approach and were specifically designed for regression problems with zero-sum constraint,
while \QP\ uses the default algorithm implemented in Gurobi for quadratic programs, i.e., the barrier algorithm.

The results obtained from the comparisons are described in Subsection~\ref{subsec:synt_res} and~\ref{subsec:real_res}.
Finally, in Subsection~\ref{subsec:grid_res} we show how a warm start strategy can be used in \ASZSL\
to solve a sequence of problems with decreasing regularization parameters.

\subsection{Synthetic datasets}\label{subsec:synt_res}
In the first experiments, we generated some synthetic datasets for log-contrast model as suggested in~\cite{lin:2014}, using the function \textit{comp\_Model} implemented in the Compack package.
More specifically, a matrix $M \in \R^{m \times n}$ was first generated from a multivariate normal distribution $N(\omega,\Sigma)$.
To model the presence of five major components in the composition,
the vector $\omega$ has all zeros except for $\omega_i = \log(0.5n)$, $i=1,\ldots,5$.
The matrix $\Sigma$ has $0.5^{|i-j|}$ in position $(i,j)$.
Then, the log-contrast model~\eqref{log_contrast_model} was obtained by setting the matrix $A$ such that
$A_{ij} = \log(Z_{ij})$ in position $(i,j)$, with
\[
Z_{ij} = \frac{\displaystyle{e^{M_{ij}}}}{\displaystyle{\sum_{h=1}^n e^{M_{ih}}}}, \quad i = 1,\ldots,m, \quad j = 1,\ldots,n,
\]
the vector of regression coefficients $x = (1, -0.8, 0.6, 0, 0, -1.5, -0.5, 1.2, 0, \ldots, 0)^T$
and the noise terms in $\eps$ were generated from a normal distribution $N(0,0.5^2)$.
We see that $x$ has six non-zero coefficients, with three of them being among the five major components.

We considered problems with dimensions $m = 2000$ and $n \in \{m,2m,5m\}$.
For every pair $(m,n)$ we generated $10$ different datasets and, for each of them, we used $\lambda \in \{\lambda_1,\ldots,\lambda_5\}$
such that $\lambda_1,\ldots,\lambda_5$ are logarithmically equally spaced between $10^{\lambda_1}$ and $10^{\lambda_5}$,
where
\[
\lambda_1 = 0.95 \lambda_{\text{max}}, \quad \lambda_5 = 10^{-3} \lambda_{\text{max}}
\]
and $\lambda_{\text{max}}$ is such that $x^* = 0$ is an optimal solution of problem~\eqref{prob}
if and only if $\lambda \ge \lambda_{\text{max}}$.
We can easily compute $\lambda_{\text{max}}$ from point~\ref{th_opt_c} of Theorem~\ref{th:opt}
(it also follows from Corollary~1 of~\cite{jeon:2020}):
\[
\lambda_{\text{max}} = \frac{\displaystyle{\max_{j=1,\ldots,n} \, [(A^T)y]_j - \min_{i=1\ldots,n} \, [(A^T)y]_i}}2.
\]

In Table~\ref{tab:res_synt1},
we show the results obtained with the different problem dimensions.
In particular, for each considered $\lambda$, we report the average values over the 10 runs
in terms of final objective value and CPU time.
We see that \ASZSL\ always took less $2$ seconds on average to solve all the considered problems,
being much faster than the other methods and also achieving the lowest objective function value.
In particular, \ASZSL\ is one or two orders of magnitude faster than the other coordinate descent based methods,
i.e., \compCL\ and \zeroSum.

{\setlength{\tabcolsep}{0.13em}
\begin{table}[t]
\scriptsize
\caption{Results on synthetic datasets for log-contrast models with six non-zero regression coefficients.
The final objective value is indicated by $f^*$, while the CPU time in seconds is indicated by \textit{time}.}
\centering
{\begin{tabular}{ c c c c c c c c c c c }
\\
\multicolumn{11}{c}{$m = 2000$, $n = 2000$} \\
\midrule
& \multicolumn{2}{c}{$\lambda_1$} & \multicolumn{2}{c}{$\lambda_2$} & \multicolumn{2}{c}{$\lambda_3$} &
  \multicolumn{2}{c}{$\lambda_4$} & \multicolumn{2}{c}{$\lambda_5$} \\
  \cmidrule(lr){2-3} \cmidrule(lr){4-5} \cmidrule(lr){6-7} \cmidrule(lr){8-9} \cmidrule(lr){10-11}
& $f^*$ & time & $f^*$ & time & $f^*$ & time & $f^*$ & time & $f^*$ & time \\
\midrule
\ASZSL & 4.70e+04 & 0.02 & 1.61e+04 & 0.03 & 4.71e+03 & 0.04 & 1.59e+03 & 0.20 & 5.21e+02 & 0.38 \\
\compCL & 4.70e+04 & 2.10 & 1.61e+04 & 1.44 & 4.71e+03 & 1.44 & 1.59e+03 & 5.02 & 5.37e+02 & 5.89 \\
\QP & 4.70e+04 & 7.38 & 1.61e+04 & 7.06 & 4.71e+03 & 6.40 & 1.59e+03 & 7.47 & 5.21e+02 & 7.14 \\
\zeroSum & 4.71e+04 & 2.53 & 4.71e+04 & 2.54 & 2.79e+04 & 2.55 & 4.61e+03 & 2.59 & 2.01e+03 & 3.31 \\
\midrule \\
\multicolumn{11}{c}{$m = 2000$, $n = 4000$} \\
\midrule
& \multicolumn{2}{c}{$\lambda_1$} & \multicolumn{2}{c}{$\lambda_2$} & \multicolumn{2}{c}{$\lambda_3$} &
  \multicolumn{2}{c}{$\lambda_4$} & \multicolumn{2}{c}{$\lambda_5$} \\
  \cmidrule(lr){2-3} \cmidrule(lr){4-5} \cmidrule(lr){6-7} \cmidrule(lr){8-9} \cmidrule(lr){10-11}
& $f^*$ & time & $f^*$ & time & $f^*$ & time & $f^*$ & time & $f^*$ & time \\
\midrule
\ASZSL & 5.41e+04 & 0.04 & 1.84e+04 & 0.09 & 5.19e+03 & 0.09 & 1.79e+03 & 0.36 & 5.66e+02 & 1.22 \\
\compCL & 5.41e+04 & 4.30 & 1.84e+04 & 2.77 & 5.19e+03 & 2.82 & 1.79e+03 & 9.98 & 5.85e+02 & 11.49 \\
\QP & 5.41e+04 & 104.77 & 1.84e+04 & 96.41 & 5.19e+03 & 80.44 & 1.79e+03 & 105.43 & 5.66e+02 & 97.74 \\
\zeroSum & 5.42e+04 & 9.97 & 5.42e+04 & 9.93 & 3.85e+04 & 10.03 & 5.51e+03 & 10.02 & 2.16e+03 & 11.47 \\
\midrule \\
\multicolumn{11}{c}{$m = 2000$, $n = 10000$} \\
\midrule
& \multicolumn{2}{c}{$\lambda_1$} & \multicolumn{2}{c}{$\lambda_2$} & \multicolumn{2}{c}{$\lambda_3$} &
  \multicolumn{2}{c}{$\lambda_4$} & \multicolumn{2}{c}{$\lambda_5$} \\
  \cmidrule(lr){2-3} \cmidrule(lr){4-5} \cmidrule(lr){6-7} \cmidrule(lr){8-9} \cmidrule(lr){10-11}
& $f^*$ & time & $f^*$ & time & $f^*$ & time & $f^*$ & time & $f^*$ & time \\
\midrule
\ASZSL & 6.50e+04 & 0.10 & 2.18e+04 & 0.19 & 5.81e+03 & 0.24 & 2.04e+03 & 0.79 & 6.30e+02 & 1.64 \\
\compCL & 6.50e+04 & 11.30 & 2.18e+04 & 6.86 & 5.81e+03 & 7.08 & 2.04e+03 & 26.21 & 6.54e+02 & 29.84 \\
\QP & 6.50e+04 & 3239.67 & 2.18e+04 & 2537.57 & 5.81e+03 & 2364.73 & 2.04e+03 & 2738.23 & 6.30e+02 & 2745.07 \\
\zeroSum & 6.51e+04 & 61.90 & 6.51e+04 & 61.60 & 6.45e+04 & 61.92 & 7.22e+03 & 61.50 & 2.28e+03 & 63.73 \\
\bottomrule
\end{tabular}}
\label{tab:res_synt1}
\end{table}}

Next, we used the same datasets described above, but with the vector of regression coefficients $x$ containing $5\%$ of randomly chosen non-zero entries,
which were generated from a uniform distribution in $(-1,1)$.

The results are shown in Table~\ref{tab:res_synt2}.
Also in this case, we see that \ASZSL\ achieves the lowest objective function value and is the fastest method,
except for the largest problems with $\lambda_5$, where \compCL\ took less time, but it returned a higher objective function value.

{\setlength{\tabcolsep}{0.13em}
\begin{table}[t]
\scriptsize
\caption{Results on synthetic datasets for log-contrast models with $5\%$ of non-zero regression coefficients.
The final objective value is indicated by $f^*$, while the CPU time in seconds is indicated by \textit{time}.}
\centering
{\begin{tabular}{ c c c c c c c c c c c }
\\
\multicolumn{11}{c}{$m = 2000$, $n = 2000$} \\
\midrule
& \multicolumn{2}{c}{$\lambda_1$} & \multicolumn{2}{c}{$\lambda_2$} & \multicolumn{2}{c}{$\lambda_3$} &
  \multicolumn{2}{c}{$\lambda_4$} & \multicolumn{2}{c}{$\lambda_5$} \\
  \cmidrule(lr){2-3} \cmidrule(lr){4-5} \cmidrule(lr){6-7} \cmidrule(lr){8-9} \cmidrule(lr){10-11}
& $f^*$ & time & $f^*$ & time & $f^*$ & time & $f^*$ & time & $f^*$ & time \\
\midrule
\ASZSL & 3.46e+04 & 0.03 & 1.85e+04 & 0.12 & 6.03e+03 & 0.27 & 1.43e+03 & 0.55 & 2.71e+02 & 1.48 \\
\compCL & 3.46e+04 & 4.18 & 1.94e+04 & 10.21 & 8.39e+03 & 10.64 & 3.40e+03 & 9.83 & 1.52e+03 & 9.16 \\
\QP & 3.46e+04 & 6.85 & 1.85e+04 & 6.74 & 6.03e+03 & 9.69 & 1.43e+03 & 6.50 & 2.71e+02 & 6.25 \\
\zeroSum & 3.46e+04 & 2.43 & 2.02e+04 & 3.08 & 7.71e+03 & 3.95 & 1.49e+03 & 5.18 & 2.73e+02 & 40.79 \\
\midrule \\
\multicolumn{11}{c}{$m = 2000$, $n = 4000$} \\
\midrule
& \multicolumn{2}{c}{$\lambda_1$} & \multicolumn{2}{c}{$\lambda_2$} & \multicolumn{2}{c}{$\lambda_3$} &
  \multicolumn{2}{c}{$\lambda_4$} & \multicolumn{2}{c}{$\lambda_5$} \\
  \cmidrule(lr){2-3} \cmidrule(lr){4-5} \cmidrule(lr){6-7} \cmidrule(lr){8-9} \cmidrule(lr){10-11}
& $f^*$ & time & $f^*$ & time & $f^*$ & time & $f^*$ & time & $f^*$ & time \\
\midrule
\ASZSL & 8.52e+04 & 0.05 & 5.36e+04 & 0.31 & 3.22e+04 & 0.69 & 1.14e+04 & 2.20 & 2.33e+03 & 5.84 \\
\compCL & 8.52e+04 & 17.26 & 5.51e+04 & 19.75 & 3.50e+04 & 20.39 & 1.48e+04 & 18.10 & 4.89e+03 & 17.81 \\
\QP & 8.52e+04 & 96.86 & 5.36e+04 & 103.67 & 3.22e+04 & 95.91 & 1.14e+04 & 99.98 & 2.33e+03 & 96.76 \\
\zeroSum & 8.52e+04 & 9.78 & 6.75e+04 & 12.09 & 4.45e+04 & 16.19 & 1.19e+04 & 22.96 & 2.35e+03 & 149.59 \\
\midrule \\
\multicolumn{11}{c}{$m = 2000$, $n = 10000$} \\
\midrule
& \multicolumn{2}{c}{$\lambda_1$} & \multicolumn{2}{c}{$\lambda_2$} & \multicolumn{2}{c}{$\lambda_3$} &
  \multicolumn{2}{c}{$\lambda_4$} & \multicolumn{2}{c}{$\lambda_5$} \\
  \cmidrule(lr){2-3} \cmidrule(lr){4-5} \cmidrule(lr){6-7} \cmidrule(lr){8-9} \cmidrule(lr){10-11}
& $f^*$ & time & $f^*$ & time & $f^*$ & time & $f^*$ & time & $f^*$ & time \\
\midrule
\ASZSL & 1.73e+05 & 0.15 & 1.18e+05 & 0.91 & 4.85e+04 & 3.52 & 1.40e+04 & 18.53 & 2.86e+03 & 95.71 \\
\compCL & 1.73e+05 & 34.30 & 1.21e+05 & 46.58 & 5.41e+04 & 53.06 & 2.01e+04 & 40.07 & 5.68e+03 & 30.30 \\
\QP & 1.73e+05 & 2622.73 & 1.18e+05 & 2890.54 & 4.85e+04 & 2963.40 & 1.40e+04 & 3216.94 & 2.86e+03 & 3586.99 \\
\zeroSum & 1.73e+05 & 61.94 & 1.23e+05 & 92.17 & 5.47e+04 & 206.69 & 1.52e+04 & 700.20 & 3.24e+03 & 1930.67 \\
\bottomrule
\end{tabular}}
\label{tab:res_synt2}
\end{table}}

\subsection{Real datasets}\label{subsec:real_res}
Now we show the results obtained on real microbiome data, which are generally regarded as compositional in the literature
(see, e.g.,~\cite{combettes:2021,gloor:2017,mishra:2022,shi:2016}).
We used three datasets considered in~\cite{quinn:2020}, 
downloaded from~\url{https://github.com/nphdang/DeepCoDA} and containing data from~\cite{vangay:2019}
suitably adjusted to deal with log-contrast model by proper replacement of the zero features (which are not allowed in such models).
The considered datasets are for binary classification (then, responses are in $\{0,1\}$) and are described in Table~\ref{tab:real_data}.

\begin{table}
\scriptsize
\caption{Microbiome datasets from~\cite{quinn:2020,vangay:2019}.}
\centering
{\begin{tabular}{ c c c c c }
\toprule
Dataset & $m$ & $n$ & Class 1 & Class 2 \\
\midrule
1 & 2070 & 3090 & Gastro & Oral \\
2 & 404 & 3090 & Stool & Tongue \\
3 & 408 & 3090 & Subgingival & Supragingival \\
\bottomrule
\end{tabular}}
\label{tab:real_data}
\end{table}

For each dataset, first we applied a log transformation and then we chose $\lambda$ by a 5-fold cross validation.
%
The final results are shown in Table~\ref{tab:res_real}.
We see that \ASZSL\ took less than $1$ second on all the problems, still being the fastest method
and achieving the lowest objective function value.

\begin{table}[h]
\scriptsize
\caption{Results on microbiome datasets.
The final objective value is indicated by $f^*$, while the CPU time in seconds is indicated by \textit{time}.}
\centering
{\begin{tabular}{ c c c c c c c }
\toprule
& \multicolumn{2}{c}{Dataset 1} & \multicolumn{2}{c}{Dataset 2} & \multicolumn{2}{c}{Dataset 3} \\
  \cmidrule(lr){2-3} \cmidrule(lr){4-5} \cmidrule(lr){6-7}
& $f^*$ & time & $f^*$ & time & $f^*$ & time \\
\midrule
\ASZSL & 3.38e+01 & 0.48 & 2.84e+01 & 0.20 & 7.18e+01 & 0.08 \\
\compCL & 3.38e+01 & 13.86 & 2.95e+01 & 2.61 & 7.18e+01 & 4.39 \\
\QP & 3.38e+01 & 95.63 & 2.84e+01 & 57.50 & 7.18e+01 & 81.98 \\
\zeroSum & 2.56e+02 & 7.27 & 7.35e+01 & 1.77 & 1.05e+02 & 0.97 \\
\bottomrule
\end{tabular}}
\label{tab:res_real}
\end{table}

{\setlength{\tabcolsep}{0.2em}
\begin{table}[h]
\scriptsize
\caption{Comparison of CPU time required by \ASZSL\ with and without warm start over $10$ values of $\lambda$.
All results are in seconds.}
\centering
{\begin{tabular}{ c c c c c c c c c c c }
\toprule
& $\lambda_1$ & $\lambda_2$ & $\lambda_3$ & $\lambda_4$ & $\lambda_5$ & $\lambda_6$ & $\lambda_7$ & $\lambda_8$ & $\lambda_9$ & $\lambda_{10}$ \\
\midrule
time with warm start & 0.14 & 0.46 & 0.99 & 1.77 & 2.95 & 4.87 & 7.83 & 12.34 & 15.60 & 19.45 \\
time without warm start & 0.14 & 0.44 & 0.80 & 1.29 & 2.79 & 5.24 & 10.94 & 24.69 & 52.16 & 104.07 \\
\midrule
cumulative time with warm start & 0.14 & 0.60 & 1.59 & 3.36 & 6.31 & 11.18 & 19.01 & 31.35 & 46.95 & 66.40 \\
cumulative time without warm start & 0.14 & 0.58 & 1.38 & 2.67 & 5.46 & 10.69 & 21.63 & 46.32 & 98.48 & 202.55 \\
\bottomrule
\end{tabular}}
\label{tab:res_grid}
\end{table}}

\subsection{Optimizing over a grid of regularization parameters}\label{subsec:grid_res}
In the previous experiments, we have shown the performances of~\ASZSL\ on several instances of problems~\eqref{prob}
for a specific value of the regularization parameter $\lambda$.
Now, we want to analyze a simple \textit{warm start} strategy for~\ASZSL\ to solve a sequence of problems
with decreasing regularization parameters.
This can be useful in practice when a suitable value of $\lambda$ is not known in advance and a parameter selection procedure must be carried out.

We generated $10$ synthetic datasets as explained in Subsection~\ref{subsec:synt_res},
with $m = 2000$, $n = 10000$ and $5\%$ of non-zero entries in the vector of regression coefficients.
For each dataset, we considered $10$ different parameters $\lambda_1,\ldots,\lambda_{10}$
logarithmically equally spaced between $10^{\lambda_1}$ and $10^{\lambda_{10}}$,
where $\lambda_1 = 0.95 \lambda_{\text{max}}$, $\lambda_{10} = 10^{-3} \lambda_{\text{max}}$
and $\lambda_{\text{max}}$ was computed as explained in Subsection~\ref{subsec:synt_res}.

For $\lambda = \lambda_i$, $i = 2,\ldots,10$,
the warm start strategy simply consists in setting the starting point of \ASZSL\ as the optimal solution computed with $\lambda = \lambda_{i-1}$.

In Table~\ref{tab:res_grid}, we report the average results obtained with and without warm start.
From the third and the fourth row of the table, we observe that the warm start strategy
allowed us to complete the whole optimization process in $66$ seconds,
while it took more than $200$ seconds without warm start.
From the second and the third row of the table, we also note that running \ASZSL\ with one small $\lambda$ took more than
running \ASZSL\ several times using decreasing regularization parameters with warm start.
For example, \ASZSL\ took $104$ seconds for $\lambda = \lambda_{10}$ without warm start,
but it took a total of $66$ seconds when it was run ten times with $\lambda = \lambda_1,\ldots,\lambda_{10}$
using the warm start strategy.
This suggests that the warm start strategy might be used to further speed up the algorithm
when problem~\eqref{prob} must be solved with small values of~$\lambda$.

\section{Conclusions}\label{sec:concl}
In this paper we proposed \ASZSL, a 2-coordinate descent method with active-set estimate
to solve lasso problems with zero-sum constraint.
At every iteration, \ASZSL\ chooses between two strategies:
\MVP\ uses the whole gradient of the smooth term of the objective function and is more accurate in the active-set estimate,
while \AC\ only needs partial derivatives and is computationally more efficient.
A suitable test is used to choose between the two strategies, considering the progress in the objective function.
A theoretical analysis was carried out, showing global convergence of \ASZSL\ to optimal solutions.
We performed numerical experiments on synthetic and real datasets, showing the effectiveness of the proposed method compared
to other algorithms from the literature.
In particular, the good numerical results are due both to the identification of the zero variables carried out by the proposed active-set estimate
and to the possibility of choosing, at each iteration, between \MVP\ and \AC, in order to balance accuracy and computational efficiency.

We finally outlined a warm start strategy for \ASZSL\ to solve a sequence of problems with decreasing regularization parameters,
which may be useful when a parameter selection procedure must be carried out and many problems with different regularization parameters must be solved.

Possible directions for future research might include the extension of these results to a more general class of problems,
for example considering other loss functions, other regularization terms and other types of constraints.

\appendix
\gdef\thesection{\Alph{section}} 
\makeatletter
\renewcommand\@seccntformat[1]{Appendix \csname the#1\endcsname.\hspace{0.5em}}
\makeatother
\section{The subproblem}\label{app:subproblem}
According to~\eqref{subprob_mvp} and~\eqref{subprob_ac2cd},
every variable update in \ASZSL\ requires the resolution of a subproblem, both when using \MVP\ and when using \AC.
As highlighted in Remark~\ref{rem:subprob}, each subproblem consists in the minimization of $f$ along a direction of the form $\pm(e_i - e_j)$.
In the next proposition, we give an equivalent expression of $f$ along $\pm(e_i - e_j)$ as an univariate function, which will be useful to compute the minimizer.

\begin{proposition}\label{prop:fij}
Let $\bar x$ be any feasible point of problem~\eqref{prob}. For any $i \ne j$, we have
\[
f(\bar x + \xi (e_i-e_j)) = \fijx ij{\bar x}(\bar x_i + \xi), \quad \forall \xi \in \R,
\]
where $\fijx ij{\bar x} \colon \R \to \R$ is the function defined as follows:
\[
\fijx ij{\bar x}(u) = \frac 12 \alpha u^2 - \beta u + \lambda(|u|+|u - \bar x_i - \bar x_j|) + c,
\]
with
\begin{align*}
\alpha & = \norm{A^i - A^j}^2, \\
\beta & = \alpha \bar x_i - \nabla_i \f0(\bar x) + \nabla_j \f0(\bar x), \\
c & = \frac 12 \norm{y - A \bar x + (A^i-A^j) \bar x_i}^2 + \lambda \sum_{t \ne i,j} |\bar x_t|.
\end{align*}
\end{proposition}

\begin{proof}
First, let us write the function $f$ as follows:
\[
f(x) = \frac 12 \sum_{h=1}^m \biggl(A_{hi} x_i + A_{hj} x_j + \sum_{t \ne i,j} A_{ht} x_t - y_h\biggr)^2 + \lambda\biggl(|x_i|+|x_j|+\sum_{t \ne i,j} |x_t|\biggr).
\]
Now, choose any $\xi \in \R$ and let $u = \bar x + \xi (e_i-e_j)$. To prove the desired result, we have to show that
\begin{equation}\label{fw_proof}
f(u) = \fijx ij{\bar x}(u_i) = \frac 12 \alpha u_i^2 - \beta^T u_i + \lambda(|u_i|+|u_i - \bar x_i - \bar x_j|) + c.
\end{equation}
Clearly, $u_t = \bar x_t$ for all $t \ne i,j$.
Since $e^T \bar x = 0$ from the feasibility of $\bar x$, it follows that $e^T u = 0$ and
$u_j = -\sum_{t \ne i,j} u_t - u_i = -\sum_{t \ne i,j} \bar x_t - u_i = \bar x_i + \bar x_j - u_i$.
Then,
\[
\begin{split}
f(u) = & \frac 12 \sum_{h=1}^m \biggl((A_{hi} - A_{hj}) u_i + \sum_{t \ne i,j} (A_{ht}- A_{hj}) \bar x_t - y_h\biggr)^2 + \\
       & + \lambda\biggl(|u_i|+|u_i - \bar x_i - \bar x_j| + \sum_{t \ne i,j} |\bar x_t|\biggr).
\end{split}
\]
Now, let us define the vector $\rho \in \R^m$ as
\[
\rho_h = y_h - \sum_{t \ne i,j} (A_{ht}- A_{hj}) \bar x_t, \quad h = 1,\ldots,m,
\]
so that
\[
\begin{split}
f(u) & = \frac 12 \norm{(A^i - A^j) u_i - \rho}^2 + \lambda\biggl(|u_i|+|u_i - \bar x_i - \bar x_j| + \sum_{t \ne i,j} |\bar x_t|\biggr) \\
     & = \frac 12 \alpha^2 u_i^2 - \rho^T (A^i-A^j) u_i + \frac 12\|\rho\|^2
         + \lambda\biggl(|u_i|+|u_i - \bar x_i - \bar x_j| + \sum_{t \ne i,j} |\bar x_t|\biggr).
\end{split}
\]
For all $h = 1,\ldots,m$, we can write
\[
\begin{split}
\rho_h & = y_h - \sum_{t \ne i,j}A_{ht} \bar x_t + A_{hj} \sum_{t \ne i,j} \bar x_t
         = y_h - (A \bar x)_h + A_{hi} \bar x_i + A_{hj} \bar x_j + A_{hj} \sum_{t \ne i,j} \bar x_t, \\
       & = y_h - (A \bar x)_h + (A_{hi} - A_{hj}) \bar x_i,
\end{split}
\]
where the last inequality follows from the fact that $\bar x_j = -\sum_{t \ne i,j} \bar x_t - \bar x_i$, since $e^T \bar x = 0$. Therefore,
\begin{equation}\label{betabar}
\rho = y - A \bar x + (A^i-A^j) \bar x_i
\end{equation}
and we obtain
\[
f(u) = \frac 12 \alpha u_i^2 - \rho^T (A^i-A^j) u_i + \lambda\biggl(|u_i|+|u_i - \bar x_i - \bar x_j|\biggr) + c.
\]
Finally, since $(A \bar x - y)^T A^i = \nabla_i \f0(\bar x)$ and $(A \bar x - y)^T A^j = \nabla_j \f0(\bar x)$, from~\eqref{betabar} it follows that
$\rho^T (A^i-A^j) = \|A^i-A^j\|^2 \bar x_i - \nabla_i \f0(\bar x) + \nabla_j \f0(\bar x) = \alpha \bar x_i - \nabla_i \f0(\bar x) + \nabla_j \f0(\bar x) = \beta$,
thus proving~\eqref{fw_proof}.
\end{proof}

\subsection{Ensuring strict convexity}\label{app:strict_conv}
In the convergence analysis of \ASZSL, we require $f$ to be strictly convex along $\pm(e_i - e_j)$
for every index pair $(i,j)$ selected by \MVP\ or \AC\ (see Subsection~\ref{subsec:conv}).
Using Proposition~\ref{prop:fij}, it is now easy to show that this requirement is satisfied if and only if $A^i \ne A^j$.

\begin{corollary}\label{corol:fij}
Let $\bar x$ be a feasible point for problem~\eqref{prob}.
For any $i \ne j$, we have that $f$ is strictly convex over the line $\{\bar x + \xi (e_i-e_j)), \, \xi \in \R\}$ if and only if $A^i \ne A^j$.
\end{corollary}

\begin{proof}
The result follows from Proposition~\ref{prop:fij}, observing that $\fijx ij{\bar x}$ is strictly convex if and only if $A^i \ne A^j$
(from the expression of $\alpha$) and that any pair of distinct points $x',x''$ over the line $\{\bar x + \xi (e_i-e_j)), \, \xi \in \R\}$
can be expressed as $x' = \bar x + \xi' (e_i-e_j)$ and $x'' = \bar x + \xi'' (e_i-e_j)$, for some $\xi' \ne \xi''$.
%
%
\end{proof}

When $A^i = A^j$, in the next proposition we show the variable $x_i$ can be safely removed from the problem.

\begin{proposition}\label{prop:x_i_zero}
Assume that $A^i = A^j$ for some $i \ne j$. If $x^*$ is an optimal solution of $\min \{f(x) \colon e^T x = 0, x_i = 0\}$,
then $x^*$ is an optimal solution of problem~\eqref{prob} as well.
\end{proposition}

\begin{proof}
By contradiction, assume that $x^*$ is not an optimal solution of problem~\eqref{prob}. Then, there exists a feasible point $x'$ for problem~\eqref{prob}
such that $f(x') < f(x^*)$.
Now, let us define $x''$ as follows:
\[
x''_h =
\begin{cases}
x'_h, & \quad \text{if } h \in \{1,\ldots,n\} \setminus \{i,j\}, \\
0, & \quad \text{if } h = i, \\
x'_i + x'_j, & \quad \text{if } h = j.
\end{cases}
\]
Clearly, $e^T x'' = e^T x' = 0$. Since $A^i = A^j$, we have $A x'' = A x'$ and, using the triangular inequality, $\norm{x''}_1 \le \norm{x'}_1$.
It follows that $f(x'') \le f (x') < f(x^*)$, contradicting the fact that $x^*$ is an optimal solution of $\min \{f(x) \colon e^T x = 0, x_i = 0\}$.
\end{proof}

In conclusion,
Proposition~\ref{prop:x_i_zero} and Corollary~\ref{corol:fij} suggest a simple procedure to ensure that, after a finite number of iterations,
$f$ is strictly convex along $\pm(e_i - e_j)$ for every index pair $(i,j)$ selected by \MVP\ or \AC.
Namely, if we find two identical columns $A^i$ and $A^j$, we can simply fix $x_i = 0$ and remove this variable from problem~\eqref{prob} (together with the column $A^i$).
We note that checking if two columns $A^i$ and $A^j$ are identical does not require additional computational burden because, as explained below,
$\norm{A^i-A^j}$ must be computed anyway
(in order to calculate the coefficients $\alpha$ and $\beta$ appearing in Proposition~\ref{prop:fij}).

\subsection{Computing the optimal solution}\label{app:solutions}
Given a feasible point $\bar x$ of problem~\eqref{prob} and an index pair $(i,j)$, with $i \ne j$,
now we show how to compute
\[
\begin{split}
& \hat x = \bar x + \xi^*(e_i-e_j), \\
& \xi^* \in \argmin_{\xi \in \R} \{f(\bar x + \xi(e_i-e_j))\}.
\end{split}
\]
According to~\eqref{subprob_mvp} and~\eqref{subprob_ac2cd}, a computation of this form is needed for the variable update
both when using \MVP\ and when using \AC.
Note that $\hat x$ can be equivalently obtained as an optimal solution of the following problem:
\begin{equation}\label{subprob}
\begin{split}
& \min f(x) \\
& x \in \{\bar x + \xi (e_i-e_j), \, \xi \in \R\}.
\end{split}
\end{equation}
So, in view of Proposition~\ref{prop:fij}, we can calculate
\[
u^* \in \argmin \fijx ij{\bar x}(u)
\]
and set
\[
\hat x_h =
\begin{cases}
u^*, \quad & \text{if }h = i, \\
\bar x_h, \quad & \text{if } h \ne i,j, \\
 -\sum_{t \ne j} \hat x_t, \quad & \text{if } h = j.
\end{cases}
\]
To compute $u^*$, let us first recall that, from Proposition~\ref{prop:fij}, we have
\begin{equation}\label{fij}
\fijx ij{\bar x}(u) = \frac 12 \alpha u^2 - \beta u + \lambda(|u|+|u - \bar x_i - \bar x_j|) + c,
\end{equation}
where $\alpha = \norm{A^i - A^j}^2$, $\beta = \alpha \bar x_i - \nabla_i \f0(\bar x) + \nabla_j \f0(\bar x)$
and $c$ is a constant.
Since $\fijx ij{\bar x}$ is coercive, then it has a minimizer.
If $\alpha = 0$ (i.e., if $A^i = A^j$), we explained above that we can simply fix $x_i = 0$
and remove this variable from the problem.
So, here we only focus on $\alpha>0$.
In this case, observe that $\fijx ij{\bar x}$ is strictly convex, has a unique minimizer $u^*$ and we can write
\[
\fijx ij{\bar x}(u) =
\begin{cases}
\frac 12 \alpha u^2 - \beta u + 2 \lambda u - \lambda (\bar x_i + \bar x_j) + c, \quad & \text{if } u \ge 0 \text{ and } u \ge \bar x_i - \bar x_j, \\
\frac 12 \alpha u^2 - \beta u - 2 \lambda u + \lambda (\bar x_i + \bar x_j) + c, \quad & \text{if } u \le 0 \text{ and } u \le \bar x_i - \bar x_j, \\
\frac 12 \alpha u^2 - \beta u + \lambda (\bar x_i + \bar x_j) + c, \quad & \text{if } u \ge 0 \text{ and } u < \bar x_i - \bar x_j, \\
\frac 12 \alpha u^2 - \beta u - \lambda (\bar x_i + \bar x_j) + c, \quad & \text{if } u \le 0 \text{ and } u > \bar x_i - \bar x_j.
\end{cases}
\]
Hence, we can first seek a stationary point of $\fijx ij{\bar x}$ where the function is differentiable, i.e., in $\R \setminus \{0,\bar x_i + \bar x_j\}$.
If such a stationary point exists, then it will be the desired minimizer $u^*$.
Otherwise, $u^*$ will be a point of non-differentiability, that is, either $0$ or $\bar x_i + \bar x_j$.
This procedure is reported in Algorithm~\ref{alg:subprob}.

\begin{algorithm}
\footnotesize
\caption{\texttt{\hspace*{0.1truecm}to compute the minimizer $u^*$ of~\eqref{fij}}}
\label{alg:subprob}
\begin{algorithmic}
\par\vspace*{0.1cm}
\item[]\hspace*{-0.1truecm}$\,\,\,\,\,\,$\vspace{0.1truecm}\hspace*{0.1truecm} \textit{seek a stationary point}
\item[]\hspace*{-0.1truecm}$\,\,\,0$\vspace{0.1truecm}\hspace*{0.1truecm} Set $\bar u = \dfrac{\beta-2\lambda}{\alpha}$
\item[]\hspace*{-0.1truecm}$\,\,\,1$\vspace{0.1truecm}\hspace*{0.1truecm} \textbf{If} $\bar u > \max\{\bar x_i+\bar x_j,0\}$
\item[]\hspace*{-0.1truecm}$\,\,\,2$\vspace{0.1truecm}\hspace*{0.6truecm} Set $u^* = \bar u$ and EXIT (\textit{stationary point found})
\item[]\hspace*{-0.1truecm}$\,\,\,3$\vspace{0.1truecm}\hspace*{0.1truecm} \textbf{Else}
\item[]\hspace*{-0.1truecm}$\,\,\,4$\vspace{0.1truecm}\hspace*{0.6truecm} Set $\bar u = \dfrac{\beta+2\lambda}{\alpha}$
\item[]\hspace*{-0.1truecm}$\,\,\,5$\vspace{0.1truecm}\hspace*{0.6truecm} \textbf{If} $\bar u < \min\{\bar x_i+\bar x_j,0\}$
\item[]\hspace*{-0.1truecm}$\,\,\,6$\vspace{0.1truecm}\hspace*{1.1truecm} Set $u^* = \bar u$ and EXIT (\textit{stationary point found})
\item[]\hspace*{-0.1truecm}$\,\,\,7$\vspace{0.1truecm}\hspace*{0.6truecm} \textbf{Else}
\item[]\hspace*{-0.1truecm}$\,\,\,8$\vspace{0.1truecm}\hspace*{1.1truecm} Set $\bar u = \dfrac{\beta}{\alpha}$
\item[]\hspace*{-0.1truecm}$\,\,\,9$\vspace{0.1truecm}\hspace*{1.1truecm} \textbf{If} $\bar u (\bar u - \bar x_i - \bar x_j) < 0$
\item[]\hspace*{-0.1truecm}$10$\vspace{0.1truecm}\hspace*{1.6truecm} Set $u^* = \bar u$ and EXIT (\textit{stationary point found})
\item[]\hspace*{-0.1truecm}$11$\vspace{0.1truecm}\hspace*{1.1truecm} \textbf{End if}
\item[]\hspace*{-0.1truecm}$12$\vspace{0.1truecm}\hspace*{0.6truecm} \textbf{End if}
\item[]\hspace*{-0.1truecm}$13$\vspace{0.1truecm}\hspace*{0.1truecm} \textbf{End if}
\item[]
\item[]\hspace*{-0.1truecm}$\,\,\,\,\,\,$\vspace{0.1truecm}\hspace*{0.1truecm} \textit{stationary point not found, the minimizer is a point of non-differentiability}
\item[]\hspace*{-0.1truecm}$14$\vspace{0.1truecm}\hspace*{0.1truecm} \textbf{If} $\fijx ij{\bar x}(0) \le \fijx ij{\bar x}(\bar x_i + \bar x_j)$, then set $u^* = 0$
\item[]\hspace*{-0.1truecm}$15$\vspace{0.1truecm}\hspace*{0.1truecm} \textbf{Else}, set $u^* = \bar x_i + \bar x_j$
\end{algorithmic}
\end{algorithm}

Note that, in addition to $\nabla_i \f0(\bar x)$ and $\nabla_j \f0(\bar x)$,
in Algorithm~\ref{alg:subprob} we only have to compute $\norm{A^i - A^j}^2$ to get $\alpha$ and $\beta$, with a cost of $\mathcal O(m)$ operations.

\newpage
\bibliography{cristofari_2022}
\addcontentsline{toc}{section}{\refname}

\end{document}